\documentclass{amsart}

\usepackage[all]{xy}

\theoremstyle{plain}
\newtheorem{thm}{Theorem}[section]
\newtheorem{prop}[thm]{Proposition}
\newtheorem{lemma}[thm]{Lemma}

\theoremstyle{definition}
\newtheorem{dfn}[thm]{Definition}

\theoremstyle{remark}
\newtheorem{rem}[thm]{Remark}

\newcommand{\HH}{\mathrm{H}}

\begin{document}

\title{Universal deformation rings of modules over Frobenius algebras}

\author{Frauke M. Bleher}
\address{F.B.: Department of Mathematics\\University of Iowa\\
14 MacLean Hall\\Iowa City, IA 52242-1419, U.S.A.}
\email{frauke-bleher@uiowa.edu}
\thanks{The first author was supported in part by NSF Grant DMS06-51332.
The second author was supported by the Reassigned Time for Research Scholarship of the Office of Academic Affairs at Valdosta State University.}
\author{Jos\'{e} A. V\'{e}lez-Marulanda}
\address{J.V.: Department of Mathematics \& Computer Science\\Valdosta State University\\
2072 Nevins Hall\\Valdosta, GA 31698-0040, U.S.A.}
\email{javelezmarulanda@valdosta.edu}


\begin{abstract}
Let $k$ be a field, and let $\Lambda$ be a finite dimensional $k$-algebra.
We prove that if $\Lambda$ is a self-injective algebra, then every finitely generated 
$\Lambda$-module $V$ whose stable endomorphism ring is isomorphic to $k$ has a universal
deformation ring $R(\Lambda,V)$ which is a complete local commutative Noetherian $k$-algebra 
with residue field $k$. If $\Lambda$ is also a Frobenius algebra, we show that 
$R(\Lambda,V)$ is stable under taking syzygies.
We investigate a particular Frobenius algebra $\Lambda_0$ of dihedral type, as introduced by 
Erdmann, and we determine $R(\Lambda_0,V)$ for every finitely generated $\Lambda_0$-module
$V$ whose stable endomorphism ring is isomorphic to $k$.
\end{abstract}

\maketitle


\section{Introduction}
\label{s:intro}
\setcounter{equation}{0}

Let $k$ be a field of arbitrary characteristic, and let $W$ be a complete local commutative
Noetherian ring with residue field $k$. Suppose $G$ is a profinite group and $V$ is a
finite dimensional $k$-vector space with a continuous $G$-action.  
If the $G$-endomorphism ring of $V$ is
isomorphic to $k$ and $\HH^1(G,\mathrm{End}_k(V))$ is finite dimensional over $k$, 
then $V$ has a universal deformation ring $R_W(G,V)$
(see e.g. \cite{lendesmit,maz2}).
The ring $R_W(G,V)$ is a complete local commutative
Noetherian $W$-algebra with residue field $k$ which is universal with respect to 
isomorphism classes of lifts (i.e. deformations) of $V$ over such $W$-algebras.
Universal deformation rings have become an important tool in number theory, in particular if $G$ is a
profinite Galois group (see e.g. \cite{boe,cornell} and their references).
It was shown in \cite{lendesmit} that the universal deformation ring $R_W(G,V)$ is 
isomorphic to the inverse limit of the universal deformation rings
$R_W(G_i,V)$ where the $G_i$ range over the (discrete) finite quotients of $G$ through which 
the action of $G$ on $V$ factors. Thus an important case to consider with respect to
deformation rings is when $G$ itself
is a finite group, i.e. when $kG$ is a finite dimensional $k$-algebra.

In this paper, we consider the case when $W=k$ and $kG$ is replaced by a more general
finite dimensional $k$-algebra $\Lambda$.
Deformations of modules for finite dimensional algebras have been studied by many authors 
in different contexts (see e.g. \cite{ile,laudal,yau} and their references). We
focus on deformations of modules for such algebras $\Lambda$ that have properties close to
those of group algebras of finite groups. Moreover, our deformations are over arbitrary complete
local commutative Noetherian $k$-algebras with residue field $k$
(see \S\ref{s:prelim} for precise definitions).

Our main motivation to examine this case is as follows.
Suppose $G$ is a finite group and $V$ is a finitely generated $kG$-module. If
the endomorphism
ring of $V$ is isomorphic to $k$, or more generally if $V$ is indecomposable, then 
$V$ belongs to a unique block $B$ of $kG$. 
Suppose $\Lambda$ is a finite dimensional $k$-algebra that is Morita equivalent to $B$ and
$V_\Lambda$ is the $\Lambda$-module that corresponds to $V$ under this Morita equivalence.
For example, $\Lambda$ can be taken to be the basic algebra of $B$.
Then the Morita equivalence provides a correspondence between the deformations of $V$ and
the deformations of $V_\Lambda$ over all complete local commutative Noetherian 
$k$-algebras with residue field $k$, and thus
between the deformation rings $R_k(G,V)$ and $R(\Lambda,V_\Lambda)$ (for a 
precise statement, see Proposition \ref{prop:morita}). 
Hence this enables us to use methods from the representation theory of finite dimensional 
algebras to study deformation rings of group representations. 
This approach has recently led to the solution of 
various open problems. For example, this was successfully used in \cite{3sim,bc4.9,bc5}
to construct representations whose universal deformation ring is not a complete intersection,
thus answering a question posed by Flach \cite{flach} in characteristic 2.

Let now  $\Lambda$ be an arbitrary finite dimensional $k$-algebra and let $V$
be a finitely generated $\Lambda$-module. 
Using Schlessinger's criteria \cite{Sch}, it follows that $V$ always has a versal deformation ring 
$R(\Lambda,V)$ which is a complete local commutative Noetherian $k$-algebra with residue field $k$
and that $R(\Lambda,V)$ is universal if the endomorphism ring of $V$ is 
isomorphic to $k$. 
Under the assumption that $\Lambda$ is a Frobenius algebra, we obtain the following stronger result,
which is proved in \S\ref{s:prelim}.
For a more precise statement, see Theorem \ref{thm:frobenius}.

\begin{thm}
\label{thm:mainudr}
Let $\Lambda$ be a Frobenius $k$-algebra and let $V$ be a finitely generated $\Lambda$-module
whose stable endomorphism ring is isomorphic to $k$. Then the versal deformation ring
$R(\Lambda,V)$ is universal. Moreover, 
$R(\Lambda,V)$ is isomorphic to $R(\Lambda,V\oplus P)$ for every finitely generated projective
$\Lambda$-module $P$, and $R(\Lambda,V)$ is isomorphic to $R(\Lambda,\Omega(V))$ where
$\Omega(V)$ is the first syzygy of $V$. 
\end{thm}

Since group algebras of finite groups over $k$ are in
particular Frobenius algebras, this generalizes  \cite[Prop. 2.1 and Cor. 2.5]{bc} which give the 
corresponding results for group algebras.
The proof of Theorem \ref{thm:mainudr} follows the basic outline of the proofs in \cite{bc}, 
but in several instances arguments that are specific to group algebras have to be replaced
by arguments that work for arbitrary Frobenius algebras.

In \S\ref{s:particular}, we assume that $k$ is algebraically closed and we
consider the particular Frobenius algebra $\Lambda_0=kQ/I$ where 
the quiver $Q$ and the ideal $I$ of the path algebra $kQ$ are as in Figure \ref{fig:algebra}.
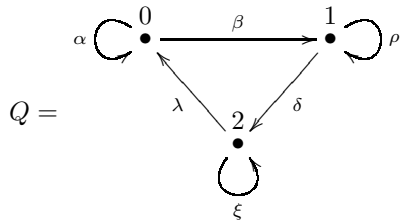
\begin{figure}[ht] \hrule \caption{\label{fig:algebra} The basic $k$-algebra $\Lambda_0=kQ/I$.}
$$Q=\vcenter{\xymatrix @R=-.2pc {
0&&1\\
\ar@(ul,dl)_{\alpha} \bullet \ar[rr]^{\beta}  &&\bullet\ar[ldddddddd]^{\delta} \ar@(ur,dr)^{\rho} 
\\&&\\&&\\&&\\&&\\&&\\&&\\&2&\\ 
&
\bullet\ar[uuuuuuuul]^{\lambda}
\ar@(ld,rd)_{\xi} & }}$$
$$I=\langle \alpha\lambda,\lambda\xi,
\xi\delta,\delta\rho,\rho\beta,\beta\alpha,
\alpha^2-\lambda\delta\beta,\rho^2-\beta\lambda\delta,\xi^2-\delta\beta\lambda\rangle.$$
\vspace{1ex}
\hrule
\end{figure}
The algebra $\Lambda_0$ belongs to the class of algebras of dihedral type which were introduced
by Erdmann in \cite{erd} to classify all tame blocks of group algebras of finite groups with
dihedral defect groups up to Morita equivalence. However, the algebra $\Lambda_0$ 
is not Morita equivalent to a block of a group algebra. Since $\Lambda_0$ is a special
biserial algebra, all the non-projective indecomposable $\Lambda_0$-modules can be described
combinatorially as so-called string and band modules (see \cite{buri} and \S \ref{s:prelimstring}). 
The components of the stable Auslander-Reiten quiver $\Gamma_s(\Lambda_0)$ consisting
of string modules are two $3$-tubes and infinitely many components of type $\mathbb{Z}A_\infty^\infty$,
whereas the components consisting of band modules are infinitely many $1$-tubes. 
If $M$ and $N$ are two indecomposable  $\Lambda_0$-modules belonging to the same component
of $\Gamma_s(\Lambda_0)$, we say $N$ is a \emph{successor} of $M$ 
if there is an irreducible homomorphism $M\to N$.

A summary of our main results concerning $\Lambda_0=kQ/I$ is as follows. The precise statements 
can be found in Theorem \ref{thm:stablend}, Propositions \ref{prop:simples}, \ref{prop:stringtype0},
\ref{prop:3tubes}, Theorem \ref{thm:stablendbands} and Proposition \ref{prop:udrbands}.
\begin{thm}
\label{thm:main}
Let $\Lambda_0=kQ/I$ be as in Figure $\ref{fig:algebra}$, and suppose $\mathfrak{C}$ is a component
of the stable Auslander-Reiten quiver $\Gamma_s(\Lambda_0)$.
\begin{itemize}
\item[(i)] If $\mathfrak{C}$ is one of the two $3$-tubes, then $\Omega(\mathfrak{C})$ is the 
other $3$-tube. There are exactly three $\Omega^2$-orbits of modules
in $\mathfrak{C}$ whose stable endomorphism ring is isomorphic to $k$. 
If $U_0$ is a module that belongs to the boundary of $\mathfrak{C}$, then these three 
$\Omega^2$-orbits are represented by $U_0$, by a successor $U_1$ of $U_0$, and by a 
successor $U_2$ of $U_1$ that does not lie in the $\Omega^2$-orbit of $U_0$. The universal
deformation rings are 
$$R(\Lambda_0,U_0)\cong R(\Lambda_0,U_1)\cong k,\qquad R(\Lambda_0,U_2)\cong k[[t]].$$

\item[(ii)] There are infinitely many components of $\Gamma_s(\Lambda_0)$ of type $\mathbb{Z}
A_\infty^\infty$ that each contain a module whose stable endomorphism ring is isomorphic to $k$.
If $\mathfrak{C}$ is such a component, then $\mathfrak{C}=\Omega(\mathfrak{C})$ and there
are exactly six  $\Omega^2$-orbits $($resp. exactly three $\Omega$-orbits$)$ of modules 
in $\mathfrak{C}$ whose stable endomorphism ring is isomorphic to $k$. If $V_0$ is a module
in $\mathfrak{C}$ of minimal length, then these three $\Omega$-orbits are represented
by $V_0$, by a successor $V_1$ of $V_0$ that does not lie in the $\Omega$-orbit of $V_0$, 
and by a successor $V_2$ of $V_1$ that does not lie in the $\Omega^2$-orbit of $V_0$. 
The universal deformation rings are 
$$R(\Lambda_0,V_0)\cong k[[t]]/(t^2),\qquad R(\Lambda_0,V_1)\cong k,\qquad R(\Lambda_0,V_2)\cong k[[t]].$$

\item[(iii)] There are infinitely many $1$-tubes of $\Gamma_s(\Lambda_0)$ that each contain a module 
whose stable endomorphism ring is isomorphic to $k$. If $\mathfrak{C}$ is such a component, 
then there is exactly one $\Omega^2$-orbit of modules in $\mathfrak{C}$ whose stable 
endomorphism ring is isomorphic to $k$, represented by a module $W_0$ belonging to the 
boundary of $\mathfrak{C}$. The universal deformation ring of $W_0$ is 
$$R(\Lambda_0,W_0)\cong k[[t]].$$
\end{itemize}
\end{thm}
Note that if $M$ is a finitely generated $\Lambda_0$-module whose stable endomorphism ring 
is isomorphic to $k$, then $M\cong U\oplus P$ where $U$ is an indecomposable $\Lambda_0$-module
whose stable endomorphism ring is isomorphic to $k$ and $P$ is a projective $\Lambda_0$-module.
Since by Theorem \ref{thm:mainudr}, $R(\Lambda_0,M)\cong R(\Lambda_0,U)$, it follows
that Theorem \ref{thm:main} describes the universal deformation ring for every finitely generated
$\Lambda_0$-module whose stable endomorphism ring is isomorphic to $k$.

Theorem \ref{thm:main} shows significant differences between $\Lambda_0$ and the blocks of 
group algebras with dihedral defect groups and precisely three isomorphism classes of simple
modules which were studied in \cite{bl,3sim}. Namely for these 
blocks, there are only finitely many components of the stable Auslander-Reiten quiver containing 
modules whose stable endomorphism rings are isomorphic to $k$. Moreover, none of the $1$-tubes 
contain such modules.

The proof of Theorem \ref{thm:main} uses the combinatorial description of the 
indecomposable $\Lambda_0$-modules and of the components of the stable Auslander-Reiten
quiver of $\Lambda_0$. 
To characterize the infinitely many components
in parts (ii) and (iii) of Theorem \ref{thm:main}, we introduce special words that start and end 
at the vertex $0$ of the quiver $Q$ and which
we call words of type $0$ (see Definition \ref{def:important}). 
Using an inductive process, we provide a precise description
of the modules $V_0$ from part (ii) (see Theorem \ref{thm:stablend} and Proposition \ref{prop:string0+})
and of the modules $W_0$ from part (iii) (see Theorem \ref{thm:stablendbands} and 
Proposition \ref{prop:bandstype0}).

Part of this paper constitutes the Ph.D. thesis of the second author under the supervision
of the first author \cite{velez}.


\section{Versal and universal deformation rings}
\label{s:prelim}
\setcounter{equation}{0}

Let $k$ be a field of arbitrary characteristic. Let $\hat{\mathcal{C}}$ be the category of 
all complete local commutative Noetherian $k$-algebras with residue field $k$. The morphisms in 
$\hat{\mathcal{C}}$ are continuous $k$-algebra homomorphisms which induce the identity map on $k$. 
Let $\mathcal{C}$ be the full subcategory of $\hat{\mathcal{C}}$ of Artinian objects.

Suppose $\Lambda$ is a finite dimensional $k$-algebra and $V$ is a finitely generated
$\Lambda$-module. A \emph{lift} of $V$ over an object $R$ in $\hat{\mathcal{C}}$ is a finitely generated 
$R\otimes_k \Lambda$-module $M$ which
is free over $R$ together with a $\Lambda$-module isomorphism $\phi:k\otimes_R M\to V$. Two lifts 
$(M,\phi)$ and $(M',\phi')$ of $V$ over $R$ are \emph{isomorphic} if there exists an 
$R\otimes_k\Lambda$-module isomorphism  $f:M\to M'$ such that 
$\phi'\circ(k\otimes_R f) = \phi$. The isomorphism class of a lift $(M,\phi)$ of $V$ 
over $R$ is denoted by $[M,\phi]$ and called a \emph{deformation} of $V$ over $R$. 
We denote the set of all such deformations over $R$ by $\mathrm{Def}_{\Lambda}(V,R)$. 
The \emph{deformation functor} $\hat{F}_V:\hat{\mathcal{C}} \to \mathrm{Sets}$ is defined
to be the following covariant functor. Let $R$ be an object in $\hat{\mathcal{C}}$ and
$\alpha:R\to R'$ a morphism in $\hat{\mathcal{C}}$. Then $\hat{F}_V(R)=\mathrm{Def}_{\Lambda}(V,R)$ 
and $\hat{F}_V(\alpha): \mathrm{Def}_{\Lambda}(V,R)\to \mathrm{Def}_{\Lambda}(V,R')$ is
defined by $\hat{F}_V(\alpha)([M,\phi]) = [R'\otimes_{R,\alpha}M,\phi_\alpha]$ where 
$\phi_\alpha:k\otimes_{R'}(R'\otimes_{R,\alpha}M) \to V$ is the composition 
$k\otimes_{R'}(R'\otimes_{R,\alpha}M) \cong k\otimes_R M\xrightarrow{\phi} V$
of $\Lambda$-modules.
Let $F_V:\mathcal{C}\to\mathrm{Sets}$ be the restriction of $\hat{F}_V$ to the subcategory
$\mathcal{C}$ of Artinian objects. 
Let $k[\epsilon]$, where $\epsilon^2=0$, denote the ring of dual numbers over $k$. 
The tangent space of $\hat{F}_V$ and of $F_V$ is defined to be the set
$t_V=F_V(k[\epsilon])$.
We say that a functor $D:\hat{\mathcal{C}}\to \mathrm{Sets}$ is 
continuous, if for all objects $R$ in $\hat{\mathcal{C}}$ whose maximal ideal is $m_R$ we have
$\displaystyle D(R)=\lim_{\stackrel{\longleftarrow}{i}} D(R/m_R^i)$.

Using Schlessinger's criteria \cite[Thm. 2.11]{Sch}, it is straightforward to prove the following result:
\begin{prop}
\label{prop:versal}
The functor $F_V$ has a pro-representable hull 
$R(\Lambda,V)\in \mathrm{Ob}(\hat{\mathcal{C}})$, 
as defined in \cite[Def. 2.7]{Sch}, 
and the functor $\hat{F}_V$ is continuous.
Moreover, there is a $k$-vector space isomorphism 
$t_V \cong\mathrm{Ext}^1_{\Lambda}(V,V)$.
If $\mathrm{End}_{\Lambda}(V)= k$, then $\hat{F}_V$ is represented
by $R(\Lambda,V)$. 
\end{prop}

\begin{rem}
\label{rem:newrem}
\begin{enumerate}
\item[(i)]
The first sentence of Proposition \ref{prop:versal} means that there exists 
a deformation $[U(\Lambda,V),\phi_U]$ of $V$ over $R(\Lambda,V)$ 
with the following property. For each $R\in \mathrm{Ob}(\hat{\mathcal{C}})$, the map
$\upsilon_R:\mathrm{Hom}_{\hat{\mathcal{C}}}(R(\Lambda,V),R) \to \hat{F}_V(R)$ 
given by $\alpha \mapsto \hat{F}_V(\alpha)([U(\Lambda,V),\phi_U])$ is surjective,
and this map is bijective if $R$ is the ring of dual numbers $k[\epsilon]$.

If $\hat{F}_V$ is represented by $R(\Lambda,V)$, in the sense that $\hat{F}_V$ is
naturally isomorphic to the functor $\mathrm{Hom}_{\Lambda}(R(\Lambda,V),-)$, then
the above map $\upsilon_R$ is bijective for all $R\in \mathrm{Ob}(\hat{\mathcal{C}})$.
\item[(ii)] The isomorphism $t_V \cong\mathrm{Ext}^1_{\Lambda}(V,V)$ is established in
the same manner as in \cite[\S22]{maz2}. Namely, given a lift $(M,\phi)$ of $V$ over
$k[\epsilon]$, 
we have $\Lambda$-module isomorphisms 
$\epsilon\cdot M\cong V$ and $M/\epsilon\cdot M\cong V$.  Thus we
obtain a short exact sequence of $\Lambda$-modules $\mathcal{E}_M:0\to V \to M \to V \to 0$.
The map $s:t_V \to \mathrm{Ext}^1_{\Lambda}(V,V)$ which sends $[M,\phi]$ to the 
class $[\mathcal{E}_M]$ in $\mathrm{Ext}^1_{\Lambda}(V,V)$ 
is well-defined. It is straightforward to see that $s$ is a $k$-vector space homomorphism. 
On the other hand, given an
extension $\mathcal{E}:0 \to V \xrightarrow{u_1} V_1\xrightarrow{u_2} V \to 0$ of $\Lambda$-modules, 
we can define a $k[\epsilon]$-module structure on $V_1$ by letting $\epsilon$ act
as $u_1\circ u_2$.  Thus there is a $\Lambda$-module
isomorphism $\psi:V_1/\epsilon V_1\to V$, and $(V_1,\psi)$ is a lift of $V$ over
$k[\epsilon]$. Sending $[\mathcal{E}]$ to $[V_1,\psi]$ defines the inverse of $s$.
\end{enumerate}
\end{rem}

\begin{dfn}
\label{def:udr}
Using the notation from Proposition \ref{prop:versal} and Remark \ref{rem:newrem}(i),
the ring $R(\Lambda,V)$  is called the \emph{versal deformation ring} of $V$ and
$[U(\Lambda,V),\phi_U]$ is called the \emph{versal deformation} of $V$. In general,
the isomorphism type of $R(\Lambda,V)$ is unique up to a (non-canonical) isomorphism.

If $R(\Lambda,V)$ represents $\hat{F}_V$, then $R(\Lambda,V)$  is called 
the \emph{universal deformation ring} of $V$ and $[U(\Lambda,V),\phi_U]$ is called the 
\emph{universal deformation} of $V$. In this case, $R(\Lambda,V)$ is unique up to 
a canonical isomorphism.
\end{dfn}

\begin{rem}
\label{rem:weakdeformations}
Note that some authors consider a weaker notion of deformations.
Namely, let $\Lambda$ and $V$ be as above, let $R$ be an object in 
$\hat{\mathcal{C}}$ and let $(M,\phi)$ be a lift of $V$ over $R$. Then the isomorphism class $[M]$ of
$M$ as an $R\otimes_k\Lambda$-module is called a \emph{weak deformation} of $V$ over $R$
(see e.g. \cite[\S 5.2]{keller}). We can also define the \emph{weak deformation functor} $\hat{F}^w_V:
\hat{\mathcal{C}}\to \mathrm{Sets}$ which sends an object $R$ in $\hat{\mathcal{C}}$ to
the set of weak deformations of $V$ over $R$ and a morphism $\alpha:R\to R'$ in $\hat{\mathcal{C}}$
to the map $\hat{F}^w_V(\alpha): \hat{F}^w_V(R)\to \hat{F}^w_V(R')$ which is defined by
$\hat{F}^w_V(\alpha)([M]) = [R'\otimes_{R,\alpha}M]$.

In general, a weak deformation of $V$ over $R$ identifies more lifts than a deformation of $V$ over $R$
that respects the isomorphism $\phi$ of a representative $(M,\phi)$.
However, if $\Lambda$ is self-injective and the stable endomorphism ring 
$\underline{\mathrm{End}}_{\Lambda}(V)$ is isomorphic to $k$, then the two deformation
functors $\hat{F}_V$ and $\hat{F}^w_V$ are naturally isomorphic (see Theorem \ref{thm:frobenius}(i)).
\end{rem}

We now prove that Morita equivalences preserve versal deformation rings.
Recall that two finitely generated $k$-algebras $\Lambda$ and $\Lambda'$ are said to be 
Morita equivalent if they have equivalent module categories. By the Morita theorems (see for example
\cite[\S3D]{CR}),  $\Lambda$ and $\Lambda'$ are Morita equivalent if and only if there
exist a $\Lambda'$-$\Lambda$-bimodule $P$ and a $\Lambda$-$\Lambda'$-bimodule $Q$
such that $P$ and $Q$ are finitely generated projective both as left and as right modules and
$Q\otimes_{\Lambda'} P\cong \Lambda$ as $\Lambda$-$\Lambda$-bimodules and
$P\otimes_\Lambda Q\cong \Lambda'$ as $\Lambda'$-$\Lambda'$-bimodules.
In this situation, we say that $P$ and $Q$ induce a Morita equivalence between $\Lambda$ and 
$\Lambda'$. Note that if $P$ and $Q$ are such bimodules, then $P\otimes_{\Lambda}-$ and
$Q\otimes_{\Lambda'}-$ give mutually inverse equivalences between the module categories of
$\Lambda$ and $\Lambda'$.

\begin{prop}
\label{prop:morita}
Let $\Lambda$ and $\Lambda'$ be finitely generated $k$-algebras. Suppose that
$P$ is a $\Lambda'$-$\Lambda$-bimodule and $Q$ is a $\Lambda$-$\Lambda'$-bimodule
which induce a Morita equivalence between $\Lambda$ and $\Lambda'$. Let $V$ be a
finitely generated $\Lambda$-module, and define $V'=P\otimes_\Lambda V$. Then the versal
deformation rings $R(\Lambda,V)$ and $R(\Lambda',V')$ are isomorphic in $\hat{\mathcal{C}}$.
\end{prop}

\begin{proof}
Suppose $R\in\mathrm{Ob}(\mathcal{C})$ is Artinian, and define $R\Lambda=R\otimes_k\Lambda$
and $R\Lambda'=R\otimes_k\Lambda'$. Then $P_R=R\otimes_k P$ is finitely generated
projective as a left $R\Lambda'$-module and as a right $R\Lambda$-module, and 
$Q_R=R\otimes_k Q$ is finitely generated projective as a left $R\Lambda$-module and as a right 
$R\Lambda'$-module. We have
$$P_R\otimes_{R\Lambda} Q_R\cong R\otimes_k(P\otimes_\Lambda Q) \cong 
R\Lambda \quad\mbox{ as $R\Lambda$-$R\Lambda$-bimodules}$$
and
$$Q_R\otimes_{R\Lambda'} P_R\cong R\otimes_k(Q\otimes_{\Lambda'} P) \cong 
R\Lambda' \quad\mbox{ as  $R\Lambda'$-$R\Lambda'$-bimodules}.$$
In particular, $P_R\otimes_{R\Lambda}-$ and
$Q_R\otimes_{R\Lambda'}-$ give mutually inverse equivalences between the module 
categories of $R\Lambda$ and $R\Lambda'$. 

Let now $(M,\phi)$ be a lift of $V$ over $R$. Then $M$ is a finitely generated 
$R\Lambda$-module. Define $M'=P_R\otimes_{R\Lambda}M$. 
Since $P_R$ is a finitely generated projective right $R\Lambda$-module and since $M$ is a finitely generated free $R$-module, it follows that $M'$ is a finitely generated projective, and hence free, $R$-module. Moreover,
\begin{equation}
\label{eq:lala}
k\otimes_RM'=k\otimes_R(P_R\otimes_{R\Lambda}M) = 
P\otimes_{\Lambda}(k\otimes_RM)\xrightarrow{P\otimes_\Lambda\phi} P\otimes_{\Lambda}V =V'.
\end{equation}
This means that $(M',\phi')=(P_R\otimes_{R\Lambda}M,P\otimes_\Lambda\phi)$ is a lift of $V'$ over $R$. 
We therefore obtain for all $R\in\mathrm{Ob}(\mathcal{C})$ a well-defined map 
$$\tau_R:\mathrm{Def}_\Lambda(V,R)\to\mathrm{Def}_{\Lambda'}(V',R).$$
Since $P_R\otimes_{R\Lambda}-$ and
$Q_R\otimes_{R\Lambda'}-$ give mutually inverse equivalences between the module 
categories of $R\Lambda$ and $R\Lambda'$, it follows that $\tau_R$ is bijective.
It is straightforward to check that $\tau_R$ is natural with respect
to homomorphisms $\alpha:R\to R'$ in $\mathcal{C}$. 
Since the deformation functors $\hat{F}_V$ and $\hat{F}_{V'}$ are continuous, this implies that they 
are naturally isomorphic. Hence the versal deformation rings $R(\Lambda,V)$ and $R(\Lambda',V')$ 
are isomorphic in $\hat{\mathcal{C}}$.
\end{proof}

In \cite{bc}, it was proved that if $\Lambda$ is the  group algebra $kG$ of a finite
group $G$, then a finitely generated $\Lambda$-module $V$ has a universal deformation
ring if the stable endomorphism ring $\underline{\mathrm{End}}_{\Lambda}(V)$ is isomorphic to $k$.
Moreover, in this case $R(\Lambda,V)\cong R(\Lambda,\Omega(V))$, where $\Omega(V)$
is the first syzygy of $V$, i.e. $\Omega(V)$ is the kernel of a projective cover $P_V\to V$. 
We want to generalize this result to arbitrary Frobenius algebras.
Recall that a finite dimensional $k$-algebra $\Lambda$ is a Frobenius algebra if and only if there exists a non-degenerate associative
bilinear form $\theta:\Lambda\times\Lambda\to k$. This is equivalent to the statement that the
right $\Lambda$-modules $\Lambda_{\Lambda}$ and $({}_{\Lambda}\Lambda)^*=
\mathrm{Hom}_k({}_{\Lambda}\Lambda,k)$ are isomorphic. In particular, the group algebra
$kG$ is a Frobenius algebra. By \cite[Prop. 9.9]{CR}, every Frobenius algebra is 
self-injective. 

\begin{thm}
\label{thm:frobenius}
Let $\Lambda$ be a finite dimensional self-injective $k$-algebra,
and suppose $V$ is a finitely generated $\Lambda$-module whose stable endomorphism ring 
$\underline{\mathrm{End}}_{\Lambda}(V)$ is isomorphic to $k$.  
\begin{itemize}
\item[(i)] The deformation functor $\hat{F}_V$ is naturally isomorphic to the weak deformation functor
$\hat{F}^w_V$ from Remark $\ref{rem:weakdeformations}$.
\item[(ii)]  The module $V$ has  a universal deformation ring $R(\Lambda,V)$.
\item[(iii)] If $P$ is a finitely generated projective $\Lambda$-module, then
$\underline{\mathrm{End}}_{\Lambda}(V\oplus P)\cong k$ and $R(\Lambda,V)\cong R(\Lambda,V\oplus P)$.
\item[(iv)] If $\Lambda$ is also a Frobenius algebra, then 
$\underline{\mathrm{End}}_{\Lambda}(\Omega(V))\cong k$ and 
$R(\Lambda,V)\cong R(\Lambda,\Omega(V))$.
\end{itemize}
\end{thm}

\begin{proof}
Suppose $\underline{\mathrm{End}}_{\Lambda}(V)\cong k$.  Then
$\underline{\mathrm{End}}_{\Lambda}(V\oplus P)\cong k$ for every finitely generated projective 
$\Lambda$-module $P$. Since $\Lambda$ is self-injective, $\Omega$ induces a self-equivalence
of the stable module category $\Lambda$-\underline{mod}. Hence also 
$\underline{\mathrm{End}}_{\Lambda}(\Omega(V))\cong k$.

We want to follow the basic outline of 
the proofs of \cite[Lemma 2.3 and
Props. 2.4 and 2.6]{bc}. 
Because the functor $\hat{F}_V$ (resp. $\hat{F}^w_V$) is continuous, most of the arguments can be 
carried out for the restriction $F_V$ (resp. $F^w_V$) of $\hat{F}_V$ (resp. $\hat{F}^w_V$) 
to the subcategory $\mathcal{C}$ of $\hat{\mathcal{C}}$ of Artinian objects.
If $R\in\mathrm{Ob}(\mathcal{C})$ is Artinian, we define $R\Lambda=R\otimes_k\Lambda$.

Let $\pi:R\to R_0$ be a surjection in $\mathcal{C}$, and let $Q_0$ be a finitely generated projective
$R_0\Lambda$-module. Since $R\Lambda$ is Artinian,
$Q_0$ has a projective $R\Lambda$-module cover $\mathrm{Proj}_R(Q_0)$, which is 
unique up to isomorphism. In particular, $R_0\otimes_{R,\pi} \mathrm{Proj}_R(Q_0)\cong Q_0$ as 
$R_0\Lambda$-modules. 
We first show in Claims 1 and 2 below
that  certain homomorphisms between $R_0\Lambda$-modules which factor through 
projective modules can be lifted to $R\Lambda$-module homomorphisms. 
Since $R,R_0$ are Artinian, we can use an inductive argument, assuming that $\pi:R\to R_0$ is a 
small extension, i.e. the kernel of $\pi$ is a principal ideal $tR$ annihilated by the maximal ideal
$m_R$ of $R$.

\medskip

\noindent\textit{Claim $1$.}
Let $\pi:R\to R_0$ be a surjection in $\mathcal{C}$.
Let $M$, $Q$ (resp. $M_0$, $Q_0$) be finitely generated $R\Lambda$-modules
(resp. $R_0\Lambda$-modules) and assume that $Q$ (resp. $Q_0$) is projective. 
Suppose there are
$R_0\Lambda$-module isomorphisms $g:R_0\otimes_{R,\pi} M \to M_0$,
$h:R_0\otimes_{R,\pi} Q\to Q_0$. 
If $\nu_0\in \mathrm{Hom}_{R_0\Lambda}(M_0, Q_0)$, then there
exists $\nu\in \mathrm{Hom}_{R\Lambda}(M, Q)$ with $\nu_0=
h\circ(R_0\otimes_{R,\pi}\nu)\circ g^{-1}$.

\medskip

\noindent
\textit{Proof of Claim $1$.}
As noted in the paragraph before the statement of Claim 1, we can assume that
$\pi:R\to R_0$ is a small extension whose kernel is $tR$.
Tensoring $0\to tR \to R \xrightarrow{\pi}R_0\to 0$ with $Q$ over $R$, we
obtain a short exact sequence of $R\Lambda$-modules
\begin{equation}
\label{eq:ses1}
0\to tQ \to Q \xrightarrow{h\circ  \tau_Q}Q_0\to 0
\end{equation}
where $\tau_Q:Q\to R_0\otimes_{R,\pi} Q$ is the natural surjection. Applying 
$\mathrm{Hom}_{R\Lambda}(M,-)$ to $(\ref{eq:ses1})$, we obtain a long exact
sequence
\begin{equation}
\label{eq:longseq}
\cdots \to \mathrm{Hom}_{R\Lambda}(M,Q)\xrightarrow{(h\circ\tau_Q)_*} 
\mathrm{Hom}_{R\Lambda}(M,Q_0)\to \mathrm{Ext}^1_{R\Lambda}(M,tQ)\to \cdots
\end{equation}
Because $tQ\cong k\otimes_R Q$ is a finitely generated projective
$\Lambda$-module, and hence an injective $\Lambda$-module, it follows that
$\mathrm{Ext}^1_{R\Lambda}(M,tQ)\cong \mathrm{Ext}^1_{\Lambda}(k\otimes_RM,tQ)=0$.
Thus the map $(h\circ\tau_Q)_*$ in $(\ref{eq:longseq})$ is surjective.
Since the natural surjection $\tau_M:M\to R_0\otimes_{R,\pi} M$ induces an isomorphism
$\mathrm{Hom}_{R_0\Lambda}(M_0,Q_0)\xrightarrow{(g\circ\tau_M)^*}
\mathrm{Hom}_{R\Lambda}(M,Q_0)$,
this implies that there exists an $R\Lambda$-module homomorphism
$\nu:M\to Q$ such that $\nu_0=h\circ(R_0\otimes_{R,\pi}\nu)
\circ g^{-1}$. This proves Claim 1.

\medskip

As noted in the paragraph before Claim 1, finitely generated projective 
$R_0\Lambda$-modules can be lifted to finitely generated projective 
$R\Lambda$-modules. We can use this together with Claim 1 to prove the following
result.

\medskip

\noindent\textit{Claim $2$.}
Let $\pi:R\to R_0$ be a surjection in $\mathcal{C}$.
Let $M$  (resp. $M_0$) be a finitely generated $R \Lambda$-module
(resp. $R_0 \Lambda$-module) such that there is an
$R_0\Lambda$-module isomorphism $g:R_0\otimes_{R,\pi} M \to M_0$.
Suppose $\sigma_0\in \mathrm{End}_{\Lambda}(M_0)$  factors through a projective 
$R_0 \Lambda$-module. Then there exists $\sigma
\in  \mathrm{End}_{R\Lambda}(M)$ such that $\sigma$ factors through a projective
$R\Lambda$-module and $\sigma_0=g\circ (R_0\otimes_{R,\pi}\sigma)\circ g^{-1}$.

\medskip

We next show in Claim 3 that if $(M,\phi)$ is a lift of $V$ over an object $R$ in
$\mathcal{C}$, then the deformation $[M,\phi]$ does not depend on the particular choice
of the $\Lambda$-module isomorphism. 
This  implies part (i) of Theorem \ref{thm:frobenius} since $\hat{F}_V$ and $\hat{F}^w_V$
are continuous. In other words,
the deformation functor $\hat{F}_V$ can be identified with the deformation functor
considered in \cite{bc}.

\medskip

\noindent
{\it Claim $3$.}
Let $R$ be an Artinian object in $\mathcal{C}$, and 
let $(M,\phi)$ and $(M',\phi')$ be two lifts of  $V$ over $R$. If there is an 
$R\Lambda$-module isomorphism $f:M\to M'$, then
there exists an $R\Lambda$-module isomorphism $f':M\to M'$ such that
$\phi'\circ(k\otimes_R f') = \phi$. In other words, $[M,\phi]=[M',\phi']$.

\medskip

\noindent{\it Proof of Claim $3$.}
Let $\overline{f}''=\phi\circ (k\otimes_R f)^{-1}\circ (\phi')^{-1}\in
\mathrm{End}_{\Lambda}(V)$. Since $\underline{\mathrm{End}}_{\Lambda}(V)=k$ by assumption,
there exists a scalar $\overline{s_f}\in k$ and
a homomorphism $\overline{\sigma}_f\in \mathrm{End}_{\Lambda}(V)$ which factors through a 
projective $\Lambda$-module such that $\overline{f}''=\overline{s_f}\cdot \mathrm{id}_V +
\overline{\sigma}_f$. 
By Claim 2,
there exists $\sigma_f
\in \mathrm{End}_{R\Lambda}(M)$ which factors through a projective 
$R\Lambda$-module such that $\overline{\sigma}_f=\phi\circ (k\otimes_R\sigma_f)
\circ \phi^{-1}$. Moreover, there exists $s_f\in R$ such that $k\otimes_R s_f= \overline{s_f}$.
Let $f''=s_f\cdot \mathrm{id}_M + \sigma_f$. Then $k\otimes_R f''=\phi^{-1}\circ \overline{f}''\circ\phi$  
is an automorphism of $k\otimes_RM$, which implies by Nakayama's Lemma that $f''$ is an 
automorphism of $M$. Define $f'=f\circ f''$. Then 
$\phi'\circ(k\otimes_R f') = \phi$, which proves Claim 3.

\medskip

We can now use Claims 2 and 3 together with analogous arguments to the ones used in the 
proof of \cite[Lemma 2.3]{bc} to prove the following result.

\medskip

\noindent
{\it Claim $4$.}
If $(M,\phi)$ is a lift of $V$ over an Artinian object $R$ in $\mathcal{C}$, then 
the ring homomorphism $\sigma_M:R\to\underline{\mathrm{End}}_{R\Lambda}(M)$
coming from the action of $R$ on $M$ via scalar multiplication is surjective.

\medskip

We next prove part (ii) of Theorem \ref{thm:frobenius}. Because of Proposition \ref{prop:versal},
it suffices to verify Schlessinger's criterion $(\mathrm{H}_4)$ in \cite[Thm. 2.11]{Sch}, which
is a consequence of the following result.

\medskip

\noindent
{\it Claim $5$.} Let $\pi: R'\to R$ be a small extension of Artinian objects in $\mathcal{C}$. 
Then the natural map $F_V(R'\times_R R') \to F_V(R')\times_{F_V(R)} F_V(R')$ is
injective.

\medskip

\noindent
\textit{Proof of Claim $5$.} Let $\alpha_1:R'\times_RR'\to R'$ (resp. $\alpha_2:R'\times_RR'\to R'$)
be the natural surjection onto the first (resp. second) component of $R'\times_RR'$ such that
$\pi\circ\alpha_1=\pi\circ\alpha_2$.
Let $(M_1,\phi_1)$ and $(M_2,\phi_2)$ be lifts of $V$ over $R'\times_RR'$.
Suppose there are $R'\Lambda$-module isomorphisms 
$f_i:R'\otimes_{R'\times_RR',\alpha_i} M_1 \to R'\otimes_{R'\times_RR',\alpha_i} M_2$ for $i=1,2$. Then 
$g_R=(R\otimes_{R',\pi}f_2)^{-1}\circ (R\otimes_{R',\pi}f_1)$ is an $R\Lambda$-module 
automorphism of  $R\otimes_{R'\times_RR',\pi\circ\alpha_2}M_1$. By Claims 2 and 4, it follows that there exists
an $R'\Lambda$-module automorphism $g$ of 
$R'\otimes_{R'\times_RR',\alpha_2}M_1$ with $R\otimes_{R',\pi}g=g_R$. Replacing $f_2$ by $f_2\circ g$,
we have $R\otimes_{R',\pi}f_1=R\otimes_{R',\pi}f_2$. Therefore, the pair $(f_1,f_2)$ defines an 
$(R'\times_RR')\Lambda$-module isomorphism $f:M_1\to M_2$. By Claim 3,
this implies that $[M_1,\phi_1]=[M_2,\phi_2]$, which proves Claim 5.

\medskip

For part (iii) of Theorem \ref{thm:frobenius}, we need the following result.

\medskip

\noindent
\textit{Claim $6$.}
Suppose $P$ is a finitely generated projective $\Lambda$-module and there is a commutative
diagram of finitely generated $R\Lambda$-modules
\begin{equation}
\label{eq:whatever}
\xymatrix{
0\ar[r]&\mathrm{Proj}_R(P)\ar[d]\ar[r]^{\alpha}&T\ar[d]\ar[r]^{\beta}& C\ar[d]\ar[r]&0\\
0\ar[r]&P\ar[r]^{\overline{\alpha}}&k\otimes_R T\ar[r]^{\overline{\beta}}&k\otimes_R C\ar[r]&0}
\end{equation}
in which $T$ and $C$ are free over $R$ and the bottom row arises by tensoring the top row with 
$k$ over $R$ and identifying $P$ with $k\otimes_R \mathrm{Proj}_R(P)$. 
Then the top row of $(\ref{eq:whatever})$ splits as a sequence of $R\Lambda$-modules.

\medskip

\noindent
\textit{Proof of Claim $6$.}
Since $\Lambda$ is self-injective, $P$ is an injective $\Lambda$-module.
Thus there exists a $\Lambda$-module homomorphism $\overline{\omega}:k\otimes_R T\to P$
with $\overline{\omega}\circ \overline{\alpha}=\mathrm{id}_P$. By Claim 1, there exists an
$R\Lambda$-module homomorphism $\omega:T\to \mathrm{Proj}_R(P)$ such that
$k\otimes_R\omega=\overline{\omega}$. Hence 
$k\otimes_R(\omega\circ\alpha) =\overline{\omega}\circ \overline{\alpha}=\mathrm{id}_P$.
Using Nakayama's Lemma, this implies that $\omega\circ\alpha$ is an 
$R\Lambda$-module automorphism of $\mathrm{Proj}_R(P)$. This proves Claim $6$.

\medskip

Part (iii) of Theorem \ref{thm:frobenius} follows now by using Claim 6 together with 
analogous arguments to the ones used in the proofs of \cite[Prop. 2.6 and Cor. 2.7]{bc}.

For the proof of part (iv) of Theorem \ref{thm:frobenius}, we assume that $\Lambda$ is a Frobenius
algebra. We need the following result.

\medskip

\noindent
{\it Claim $7$.}
Let $R$ be an Artinian object in $\mathcal{C}$, 
and let $Q$ be a finitely generated projective left $R\Lambda$-module. Then 
$Q^*=\mathrm{Hom}_{R\Lambda}(Q,R)$ is a projective right $R\Lambda$-module.

\medskip

\noindent
\textit{Proof of Claim $7$.}
Consider the additive map
\begin{eqnarray}
\label{eq:darnmap}
R\otimes_k\mathrm{Hom}_k({}_{\Lambda}\Lambda,k) &\to& 
\mathrm{Hom}_R({}_{R\otimes_k\Lambda}(R\otimes_k\Lambda),R)=
\mathrm{Hom}_R({}_{R\Lambda}(R\Lambda),R)\\
b\otimes f &\mapsto& (a\otimes \lambda \mapsto a b f(\lambda)).\nonumber
\end{eqnarray}
By \cite[Thm. 2.38]{CR}, $(\ref{eq:darnmap})$ is an $R$-module isomorphism. It is straightforward
to check that $(\ref{eq:darnmap})$ is in fact an isomorphism of right $R\Lambda$-modules.
Since $\Lambda$ is a Frobenius algebra, there is an isomorphism
$\Lambda_{\Lambda}\cong({}_{\Lambda}\Lambda)^*=\mathrm{Hom}_k({}_{\Lambda}\Lambda,k)$ 
of right $\Lambda$-modules. Therefore, 
$$\left({}_{R\Lambda}(R\Lambda)\right)^*=
\mathrm{Hom}_R({}_{R\Lambda}(R\Lambda),R)
\cong 
R\otimes_k({}_{\Lambda}\Lambda)^*
\cong 
R\otimes_k (\Lambda_{\Lambda}) \cong  (R\Lambda)_{R\Lambda}$$
as right $R\Lambda$-modules.
By assumption, $Q$ is a finitely generated projective left $R\Lambda$-module.
Hence there exists a left $R\Lambda$-module $T$ and an integer $n$ such that
$Q\oplus T\cong \left({}_{R\Lambda}(R\Lambda)\right)^n$. 
Since $Q$ and $T$ are free over $R$, they are $R\Lambda$-lattices. Thus 
we have 
$$Q^*\oplus T^*\cong \left(Q\oplus T\right)^*\cong
\left(\left({}_{R\Lambda}(R\Lambda)\right)^*\right)^n
\cong \left((R\Lambda)_{R\Lambda}\right)^n.$$
Hence $Q^*$ is a projective right $R\Lambda$-module, which proves Claim 7.

\medskip

Consider the short exact sequence $0\to \Omega(V)\xrightarrow{\overline{\varphi}} P_V\to V\to 0$
coming from the definition of $\Omega(V)$.
Let $(U,\psi)$ be a lift of $\Omega(V)$ over an Artinian object $R$ in $\mathcal{C}$. 
Since $\mathrm{Proj}_R(P_V)$  is a projective 
$R\Lambda$-module with $k\otimes_R \mathrm{Proj}_R(P_V)\cong P_V$,
we can use Claim 1 to lift $\overline{\varphi}$ 
to an $R\Lambda$-module homomorphism $\varphi:U\to \mathrm{Proj}_R(P_V)$.
Using this together with Claim 7, we can argue analogously to the proofs of
\cite[Prop. 2.4 and Cor. 2.5]{bc} to prove part (iv) of Theorem \ref{thm:frobenius}.

This completes the proof of Theorem \ref{thm:frobenius}.
\end{proof}


\section{A particular algebra of dihedral type}
\label{s:particular}
\setcounter{equation}{0}

For the remainder of this article,
let $k$ be an algebraically closed field of arbitrary characteristic and let $\Lambda_0$ be the basic
$k$-algebra $\Lambda_0=kQ/I$ where $Q$ and $I$ are as in Figure \ref{fig:algebra}.
The algebra $\Lambda_0$ is one of the algebras of dihedral type studied by Erdmann
in \cite{erd}. In particular, $\Lambda_0$ is symmetric, and hence a Frobenius algebra.
By \cite[Lemma IX.5.4]{erd}, $\Lambda_0$ is not Morita equivalent to a block of a group
algebra. 
We denote the irreducible $\Lambda_0$-modules by $S_0$, $S_1$ and $S_2$, or, using shorthand, 
by $0$, $1$ and $2$. The radical series of the projective indecomposable $\Lambda_0$-modules 
can be described by the following pictures:
\begin{equation}
\label{eq:proj}
P_0=\begin{array}{c@{\hspace*{.5ex}}c@{\hspace*{.5ex}}c}&0&\\
0&&\begin{array}{c}1\\2\end{array}\\&0&\end{array},\qquad 
P_1=\begin{array}{c@{\hspace*{.5ex}}c@{\hspace*{.5ex}}c}&1&\\ 1&&
\begin{array}{c}2\\0\end{array}\\&1&\end{array}, \qquad
P_2=\begin{array}{c@{\hspace*{.5ex}}c@{\hspace*{.5ex}}c}&2&\\ 2&&
\begin{array}{c}0\\1\end{array}\\&2&\end{array}.
\end{equation}

Since $\Lambda_0$ is a special biserial algebra, all indecomposable non-projective 
$\Lambda_0$-modules are either  string or band modules (see \cite{buri}).
We give an introduction to the representation theory
of $\Lambda_0$ in an appendix in \S \ref{s:prelimstring}, where we introduce in particular
words, strings, bands and canonical $k$-bases
for string and band modules. Moreover, we describe the components of the stable
Auslander-Reiten quiver using hooks and cohooks and
give a description of the homomorphisms between string and band
modules as determined in \cite{krau}.

Our goal in this section is to determine all string and band modules $V$ for $\Lambda_0$ whose
stable endomorphism rings are isomorphic to $k$ and to compute their universal
deformation rings $R(\Lambda_0,V)$. 
Because of the symmetries that both the quiver $Q$ and the relations $I$ satisfy, we
will make use of the following permutations of the vertices and arrows of $Q$.

\begin{dfn}
\label{def:bijections}
\begin{enumerate}
\item[(i)] For $u\in\{0,1,2\}$, define $\nu_u$ to be the following permutation of the vertices and arrows of $Q$: 
$$\nu_0=(1,2)(\rho,\xi)(\beta,\lambda),\quad
\nu_1=(0,2)(\alpha,\xi)(\beta,\delta),\quad\mbox{and}\quad
\nu_2=(0,1)(\alpha,\rho)(\lambda,\delta).$$
Note that $\nu_0$, $\nu_1$ and $\nu_2$ are not quiver automorphisms of $Q$.
If $S=w_1w_2\cdots w_n$ is a word of length $n\ge 1$ representing a string, let $u_S=e(S)$ and define
$$\nu_{u_S}(S)=\nu_{u_S}(w_1)^{-1}\nu_{u_S}(w_2)^{-1}\cdots \nu_{u_S}(w_n)^{-1}.$$
\item[(ii)] Let $\theta$ be the
quiver automorphism of $Q$ which acts on the vertices and arrows of $Q$ as the permutation
$$\theta=(0,1,2)(\alpha,\rho,\xi)(\beta,\delta,\lambda).$$
If $S=w_1w_2\cdots w_n$ is a word of length $n\ge 1$ representing a string, define
$$\theta(S)=\theta(w_1)\theta(w_2)\cdots\theta(w_n).$$
\end{enumerate}
\end{dfn}

The following words starting and ending at the vertex 0 will play a special role when discussing
both string and band modules.

\begin{dfn}
\label{def:important}
\begin{enumerate}
\item[(i)] Define $\underline{x}=\lambda\xi^{-1}\delta\beta\alpha^{-1}$ and
$\underline{y}=\lambda\delta\rho^{-1}\beta\alpha^{-1}$. 

\item[(ii)] We say a word $Z$ is a \emph{word of type $0$} if 
$Z=1_0$ or 
$$Z=\underline{z}_0 \underline{z}_1\cdots \underline{z}_{n-1}$$
where $n\ge 1$ and  
$\underline{z}_0,\ldots,\underline{z}_{n-1}\in\{\underline{x},\underline{y}\}$.
We call $n$ the \emph{type-$0$-length} of $Z$, where $n=0$ corresponds to $Z=1_0$.

\item[(iii)] Let $Z=\underline{z}_0\underline{z}_1 \cdots \underline{z}_{n-1}$ be a word of type $0$. 
A \emph{subpattern} $U$ of $Z$ is a subword
of $Z$ such that either $U=1_0$  or 
$U=\underline{z}_i \underline{z}_{i+1}\cdots  \underline{z}_j$ for some $0\le i\le  j\le n-1$.
The \emph{inverse pattern} $U^-$ of $U$ is defined to be $U^-=1_0$ if $U=1_0$ and
$U^-=\underline{z}_j \underline{z}_{j-1}\cdots\underline{z}_i$ if 
$U=\underline{z}_i \underline{z}_{i+1}\cdots  \underline{z}_j$ for some $0\le i\le  j\le n-1$.

\item[(iv)] Let $Z=\underline{z}_0 \cdots \underline{z}_{n-1}$ be a word of type $0$ of type-$0$-length 
$n\ge 1$. 
For $0\le i\le n-1$, we define the \emph{$i^{\mathrm{th}}$ rotation of $Z$ of type $0$} to be
$$\rho^{(0)}_i(Z)=\underline{z}_i\underline{z}_{i+1}\cdots \underline{z}_{n-1}\underline{z}_0\cdots 
\underline{z}_{i-1}.$$
Define $\sim_{r,(0)}$ to be the equivalence relation on all words of type $0$ of type-0-length at least 1
such that $Z\sim_{r,(0)}Z'$ if and only if $Z=\rho^{(0)}_i(Z')$ for some $i$.
\end{enumerate}
\end{dfn}


\subsection{Stable endomorphism rings and universal deformation rings of string modules for $\Lambda_0$}
\label{ss:udrstring}

In this section, we describe all string modules $M(S)$ for $\Lambda_0$ whose stable endomorphism rings are isomorphic 
to $k$ and determine their universal deformation rings.

The following result is an easy consequence of  \cite{krau} (see Remark \ref{rem:stringhoms})
and the symmetric shapes of the radical series of $P_0$, $P_1$ and $P_2$ as given in 
$(\ref{eq:proj})$.

\begin{lemma}
\label{lem:upsidedown}
Let $S=w_1w_2\cdots w_n$ be a word of length $n\ge 1$ representing a string and let $\nu_{u_S}(S)$ 
be as in Definition $\ref{def:bijections}(i)$.
Then $\nu_{u_S}(S)$ is a word representing a string. Moreover, 
\begin{eqnarray*}
\mathrm{dim}_k\,\underline{\mathrm{End}}_{\Lambda_0}(M(S)) &=&
\mathrm{dim}_k\,\underline{\mathrm{End}}_{\Lambda_0}(M(\nu_{u_S}(S))),\quad\mbox{and}\\
\mathrm{dim}_k\,\mathrm{Ext}^1_{\Lambda_0}(M(S),M(S)) &=&
\mathrm{dim}_k\,\mathrm{Ext}^1_{\Lambda_0}(M(\nu_{u_S}(S)),M(\nu_{u_S}(S))).
\end{eqnarray*}
\end{lemma}

For the description of all the string modules whose stable endomorphism
rings are isomorphic to $k$, we need to introduce special strings, which we call strings of type 
$\beta\alpha^{-1}\lambda$ and which are defined using the words of type $0$ introduced
in Definition \ref{def:important}.

\begin{dfn}
\label{def:important1}
\begin{enumerate}
\item[(i)] We say a word $S$ is the \emph{standard representative of a string of type 
$\beta\alpha^{-1}\lambda$} if
$$S=\beta\alpha^{-1}Z\,\lambda$$
where $Z$ is a word of type $0$, as defined in Definition \ref{def:important}(ii).
A subword $U$ of $S$ is called a \emph{subpattern} of $S$ if $U$ is a subpattern of $Z$,
as defined in Definition \ref{def:important}(iii). The 
\emph{inverse pattern} $U^-$ of $U$ is also defined as in Definition \ref{def:important}(iii).
We define $S^-=\beta\alpha^{-1}Z^-\lambda$.

\item[(ii)] Define $\Xi$ to be the permutation $\Xi=(\underline{x},\underline{y})$. Define $\Xi(1_0)=1_0$.
If $Z=\underline{z}_0 \cdots \underline{z}_{n-1}$ is a word of type $0$ of type-$0$-length $n\ge 1$,
define $\Xi(Z)=\Xi(\underline{z}_0) \cdots \Xi(\underline{z}_{n-1})$. If 
$S=\beta\alpha^{-1}Z\lambda$ is the standard
representative of a string of type $\beta\alpha^{-1}\lambda$, define 
$\Xi(S)=\beta\alpha^{-1}\Xi(Z)\lambda$ and $\Xi(S^-)=\beta\alpha^{-1}\Xi(Z^-)\lambda$.
\end{enumerate}
\end{dfn}

\begin{rem}
\label{rem:string0nuA}
Let $S$ be the standard representative of a string of type $\beta\alpha^{-1}\lambda$,
and let $S^-$ and $\Xi$ be as in Definition \ref{def:important1}. 
Moreover, let $\nu_0,\nu_1,\nu_2$ and $\theta$ be as in Definition \ref{def:bijections}. 
Then $\theta(\nu_1(S))=\nu_2(\theta(S)) = \Xi(S^-)^{-1}$, and hence $\Xi(S^-)=\theta^2(\nu_2(S^{-1}))$.
Moreover, $\nu_2(S^{-1})=\left(\theta^2(\nu_1(S))\right)^{-1}$.
\end{rem}

\begin{lemma}
\label{lem:needthis}
Let $S=\beta\alpha^{-1}\underline{z}_0 \cdots \underline{z}_{n-1}\,\lambda$
be the standard representative of a string of type $\beta\alpha^{-1}\lambda$ as in 
Definition $\ref{def:important1}$. 
The stable endomorphism ring of $M(S)$, $\underline{\mathrm{End}}_{\Lambda_0}(M(S))$, 
is isomorphic to $k$ if and only if $S$ has the following property:
\begin{enumerate}
\item[(+)] 
If there are subpatterns $U,A_1,A_2$ of $S$ such that
$$S=\beta\alpha^{-1}\, U \,\underline{x}\,A_2\,\lambda \;\mbox{ or } \;
S=\beta\alpha^{-1}\,A_1\,\underline{x}\,U\,\underline{x}\,A_2\,\lambda \;\mbox{ or }\;
S=\beta\alpha^{-1}\,A_1\,\underline{x}\,U\,\lambda,$$ 
then  for all subpatterns $B_1,B_2$ of $S$, we have
$$S\neq \beta\alpha^{-1}\, U \,\underline{y}\,B_2\,\lambda\;\mbox{ and }\;
S\neq \beta\alpha^{-1}\,B_1\,\underline{y}\,U\,\underline{y}\, B_2\,\lambda \;\mbox{ and }\; 
S\neq  \beta\alpha^{-1}\,B_1\,\underline{y}\,U\,\lambda.$$ 
\end{enumerate}
Moreover, if the stable endomorphism ring of $M(S)$ has $k$-dimension at least $2$,
then all modules in the stable Auslander-Reiten component containing $M(S)$ have this
property.
\end{lemma}

\begin{proof}
Since $\underline{\mathrm{End}}_{\Lambda_0}(M(\beta\alpha^{-1}\lambda)) \cong k$ 
and $\beta\alpha^{-1}\lambda$ has property $(+)$, Lemma \ref{lem:needthis} follows for $n=0$.
For the remainder of the proof, let $S=\beta\alpha^{-1}\underline{z}_0 \cdots \underline{z}_{n-1}\,\lambda$
be a word representing a string of type $\beta\alpha^{-1}\lambda$ with $n\ge 1$. 
We analyze the different possibilities for
$M(S)$ such that $\underline{\mathrm{End}}_{\Lambda_0}(M(S))$ has $k$-dimension at least $2$.

Suppose first that there exists a string $T\not\sim_s S$ of length at least $2$ such that $M(S)$ has a 
canonical endomorphism $\alpha_T$, as defined in Remark \ref{rem:stringhoms}, 
factoring through $M(T)$. 
In particular,  $\underline{\mathrm{End}}_{\Lambda_0}(M(S))$ has $k$-dimension at least $2$. 
It follows from the definition of strings of type $\beta\alpha^{-1}\lambda$ that
$$T\sim_s\beta\alpha^{-1}\,U\,\lambda$$
for some subpattern $U$ of $S$
and that there
exist subpatterns $A_1,A_2,B_1,B_2$ of $S$ satisfying both $(\ast)$ and $(\ast\ast)$, where
\begin{itemize}
\item[$(\ast)$]
$S=\beta\alpha^{-1}\, U \,\underline{x}\,A_2\,\lambda$ or 
$S=\beta\alpha^{-1}\,A_1\,\underline{x}\,U\,\underline{x}\,A_2\,\lambda$ or
$S=\beta\alpha^{-1}\,A_1\,\underline{x}\,U\,\lambda$,
\item[$(\ast\ast)$]
$S =\beta\alpha^{-1}\, U \,\underline{y}\,B_2\,\lambda$ or 
$S = \beta\alpha^{-1}\,B_1\,\underline{y}\,U\,\underline{y}\, B_2\,\lambda$ or
$S = \beta\alpha^{-1}\,B_1\,\underline{y}\,U\,\lambda$. 
\end{itemize}
On the other hand, if $S$ satisfies $(\ast)$ and $(\ast\ast)$, then $M(S)$ has a 
canonical endomorphism factoring through a string module that is not isomorphic 
to $M(S)$ with at least three composition factors.
It is straightforward to see that then also the modules $M(S_{h\cdots h})$ and $M({}_{h\cdots h}S)$
each have a non-identity
canonical endomorphism factoring through a string module 
with at least three composition factors.

Next let $T$ be a string of length $0$ or $1$ such that $M(S)$ has a 
canonical endomorphism 
$\alpha_T$ factoring through $M(T)$. Then either $T=\delta$ or $T=1_u$ for some $u\in\{0,1,2\}$.
Analyzing these cases further, it follows that $\alpha_T$ does not factor through a projective 
$\Lambda_0$-module only if $S$ satisfies $(\ast)$ and $(\ast\ast)$ for certain subpatterns 
$U,A_1,A_2,B_1,B_2$ of $S$.
This completes the proof of Lemma \ref{lem:needthis}.
\end{proof}

\begin{rem}
\label{rem:string0nuB}
Let $S$ be the standard representative of a string of type $\beta\alpha^{-1}\lambda$,
and let $S^-$ and $\Xi$ be as in Definition \ref{def:important1}. 
If $S$ satisfies condition $(+)$ from Lemma \ref{lem:needthis}, then $S=S^-$. Moreover,
$S$ satisfies condition $(+)$ if and only if $\Xi(S)$ satisfies condition $(+)$.
\end{rem}

\begin{thm}
\label{thm:stablend}
Let $\mathfrak{C}$ be a component of the stable Auslander-Reiten quiver $\Gamma_s(\Lambda_0)$
consisting of string modules. Then $\mathfrak{C}$ contains a module whose stable endomorphism 
ring is isomorphic to $k$ if and only if either
\begin{enumerate}
\item[(i)] $\mathfrak{C}$ is a $3$-tube, or
\item[(ii)] $\mathfrak{C}$ contains a simple $\Lambda_0$-module, or
\item[(iii)] there exist $i\in\{0,1,2\}$ and 
a string of type $\beta\alpha^{-1}\lambda$ whose standard representative $S$ satisfies
condition $(+)$ from Lemma $\ref{lem:needthis}$ such that
$\mathfrak{C}$ contains $M(\theta^i(S))$. 
\end{enumerate}
\end{thm}

\begin{proof}
Since one of the 3-tubes contains $M(\alpha)$ and the other 3-tube contains $M(\delta\beta)$, it follows
that if $\mathfrak{C}$ is as in (i), (ii) or (iii), then $\mathfrak{C}$ contains a module whose stable 
endomorphism ring is isomorphic to $k$.

Conversely, let $\mathfrak{C}$ be a component of string modules containing a module whose stable 
endomorphism ring is isomorphic to $k$, and suppose $\mathfrak{C}$ is not a $3$-tube.
Then $\mathfrak{C}$ is of type $\mathbb{Z}A_\infty^\infty$. Let $T$ be a word representing a string 
such that $M(T)$ has minimal length in $\mathfrak{C}$.

If $T$ has length $0$, $1$ or $2$, 
it follows that $M(T)$ lies in the same component as a simple $\Lambda_0$-module,
and hence $M(T)$ is itself simple.

Now suppose $T$ is a string of length at least 3.
Let $u_T=e(T)$, let $\nu_{u_T}$ and $\theta$ be as in 
Definition $\ref{def:bijections}$, and let $i_T\in\{0,1,2\}$ be such that $\theta^{i_T}(u_T)=1$.
Then there exists $j\in\{0,1\}$ such that 
\begin{equation}
\label{eq:form}
\tilde{S}\stackrel{\mathrm{def}}{=}\theta^{i_T}(\nu_{u_T}^j(T))=\beta\alpha^{-1}\lambda \cdots .
\end{equation}
Suppose first that $\tilde{S}$ in (\ref{eq:form}) is not the standard representative of a string of type 
$\beta\alpha^{-1}\lambda$. 
Then there exists a word $Z=\underline{z}_0\cdots \underline{z}_{n-1}$ of type $0$ of maximal
type-$0$-length $n\ge 0$ such that
$\tilde{S}=\beta\alpha^{-1} Z\,W $
for some subword $W$ of $\tilde{S}$.
In particular, $W\neq \underline{x}\cdots$ and  $W\neq \underline{y} \cdots$. 
Because $M(T)$ is assumed to have minimal length in its stable Auslander-Reiten component,
the same is true for $M(\tilde{S})$ in its stable Auslander-Reiten component. Thus $W\neq 1_0$. 
Since we assume that $\tilde{S}$ is not the standard representative of a string of type 
$\beta\alpha^{-1}\lambda$, $W\neq \lambda$. Hence 
either (a) $W=\lambda\delta\, W'$, or (b) $W=\lambda\xi^{-1}\,W'$ for some subword $W'$ of $W$. 
Analyzing these cases further, it follows that
the stable endomorphism ring of every module in the stable Auslander-Reiten 
component containing $M(\tilde{S})$ has $k$-dimension at least $2$.
Hence it follows by Lemma \ref{lem:upsidedown} that this case cannot occur.

Therefore, $\tilde{S}$ in (\ref{eq:form}) must be the standard representative 
of a string of type $\beta\alpha^{-1}\lambda$.
If $j=0$, then $\theta^{-i_T}(\tilde{S})=T$ and we define $S=\tilde{S}$. If $j=1$, then by Remark
\ref{rem:string0nuA},
$$\theta^{-i_T-1}(\Xi( \tilde{S}^-))=\theta^{-i_T-1}(\theta^2(\nu_2( \tilde{S}^{-1})))=
\theta^{-i_T+1}\left((\theta^2(\nu_1( \tilde{S})))^{-1}\right) = \left(\theta^{-i_T}(\nu_1( \tilde{S})\right)^{-1}
=T^{-1}$$
and we define $S=\Xi(\tilde{S}^-)$. By Definition \ref{def:important1}, $\Xi(\tilde{S}^-)$
is also the standard representative of a string of type $\beta\alpha^{-1}\lambda$. 
Hence in both cases $j=0$ and $j=1$, there exists $i\in\{0,1,2\}$ such that $\mathfrak{C}$ contains 
$M(\theta^i(S))$, where $S$ is the standard representative of a string of type $\beta\alpha^{-1}\lambda$. 
Since we have assumed that  $\mathfrak{C}$ contains a module whose stable endomorphism ring
is isomorphic to $k$, it follows by Lemma \ref{lem:needthis} that $S$ satisfies condition $(+)$.
This completes the proof of Theorem \ref{thm:stablend}.
\end{proof}

Because of the prominent role played by standard representatives of strings $S$ of type 
$\beta\alpha^{-1}\lambda$ in Theorem \ref{thm:stablend}, the question arises 
which such $S$ satisfy condition $(+)$ from Lemma \ref{lem:needthis}. We will give a precise
answer to this question in Proposition \ref{prop:string0+}.

We now look at the components of the stable Auslander-Reiten quiver of $\Lambda_0$ 
singled out by Theorem \ref{thm:stablend}.

\begin{prop}
\label{prop:simples}
For $i\in\{0,1,2\}$, let $\mathfrak{C}_i$ be the component of the stable Auslander-Reiten quiver of 
$\Lambda_0$ containing the simple $\Lambda_0$-module $S_i=M(1_i)$.
Define 
$(1_i)_h=\theta^i(\lambda\xi^{-1})$ and $(1_i)_{hh}=\theta^i(\lambda\xi^{-1}\delta\rho^{-1})$.
The component $\mathfrak{C}_i$ is stable under $\Omega$, and
the modules in $\mathfrak{C}_i$ whose stable endomorphism rings are isomorphic to $k$
are precisely the modules in the $\Omega$-orbits of the modules
$M(1_i)$, $M((1_i)_h)$ and $M((1_i)_{hh})$.
Their universal deformation rings are
$$R(\Lambda_0, M(1_i)) \cong k[[t]]/(t^2),\qquad R(\Lambda_0, M((1_i)_h))\cong k,
\qquad R(\Lambda_0, M((1_i)_{hh}))\cong k[[t]].$$
\end{prop}

\begin{proof}
It suffices to consider $i=0$. Using hooks and cohooks (see Definition \ref{def:arcomps}), we see that 
$\mathfrak{C}_0$ is stable under $\Omega$ and that all $\Lambda_0$-modules
in $\mathfrak{C}_0$ lie in the $\Omega$-orbit of either 
\begin{eqnarray*}
A_{n,0}&=&M\left((\lambda\xi^{-1}\delta\rho^{-1}\beta\alpha^{-1})^n\right),\mbox{ or}\\
A_{n,1}&=&M\left((\lambda\xi^{-1}\delta\rho^{-1}\beta\alpha^{-1})^n\lambda\xi^{-1}\right),\mbox{ or}\\
A_{n,2}&=&M\left((\lambda\xi^{-1}\delta\rho^{-1}\beta\alpha^{-1})^n\lambda\xi^{-1}\delta\rho^{-1}\right)
\end{eqnarray*}
for some $n\ge 0$. Note that $M(1_0)=A_{0,0}$, $M((1_0)_h)=A_{0,1}$ and $M((1_0)_{hh})=A_{0,2}$.
Using Remark \ref{rem:stringhoms} and the description of the projective indecomposable
$\Lambda_0$-modules in (\ref{eq:proj}), it is straightforward to show that the stable endomorphism ring of
$A_{0,j}$ is isomorphic to $k$ for $j\in\{0,1,2\}$ and that
$\mathrm{Ext}^1_{\Lambda_0}(A_{0,j},A_{0,j})$ is isomorphic to $k$ for $j\in\{0,2\}$ and
zero for $j=1$. On the other hand, for $n\ge 1$, $A_{n,j}$ has a non-zero endomorphism which 
factors through $M(1_0)$ and which does not factor through a projective $\Lambda_0$-module. 

Since $\mathrm{Ext}^1_{\Lambda_0}(A_{0,1},A_{0,1})=0$, it follows that 
$R(\Lambda_0,A_{0,1})\cong k$.
Since $\mathrm{Ext}^1_{\Lambda_0}(A_{0,j},A_{0,j})$ is isomorphic to $k$ for $j\in\{0,2\}$, it
follows that $R(\Lambda_0,A_{0,j})$ is a quotient of $k[[t]]$ for $j\in\{0,2\}$.

The module $A_{0,0}$ lies in a non-split short exact sequence of $\Lambda_0$-modules
\begin{equation}
\label{eq:step1}
0\to A_{0,0} \to M(\alpha) \to A_{0,0}\to 0.
\end{equation}
Therefore, $M(\alpha)$ defines a non-trivial lift of $A_{0,0}$ over $k[[t]]/(t^2)$. This implies that there
is a unique surjective $k$-algebra homomorphism $\psi:R(\Lambda_0,A_{0,0})\to k[[t]]/(t^2)$ in 
$\hat{\mathcal{C}}$ corresponding to the deformation 
defined by $M(\alpha)$. We need to show that $\psi$ is an isomorphism.
Suppose this is false. Then there exists a surjective $k$-algebra homomorphism
$\psi_0:R(\Lambda_0,A_{0,0})\to k[[t]]/(t^3)$ in $\hat{\mathcal{C}}$
such that $\pi\circ\psi_0=\psi$ where $\pi:k[[t]]/(t^3)\to
k[[t]]/(t^2)$ is the natural projection. Let $M_0$ be a $k[[t]]/(t^3)\otimes_k\Lambda_0$-module
which defines a lift of $A_{0,0}$ over $k[[t]]/(t^3)$ corresponding to $\psi_0$.
Because $M_0/t^2M_0\cong M(\alpha)$ and $t^2M_0\cong A_{0,0}$, we obtain 
a non-split short exact sequence of $k[[t]]/(t^3)\otimes_k\Lambda_0$-modules
\begin{equation}
\label{eq:step2}
0\to A_{0,0} \to M_0\to M(\alpha)\to 0.
\end{equation}
Since $\mathrm{Ext}^1_{\Lambda_0}(M(\alpha),A_{0,0})=0$, this sequence splits as a sequence
of $\Lambda_0$-modules. Hence $M_0=A_{0,0}\oplus M(\alpha)$ as $\Lambda_0$-modules.
Writing elements of $M_0$ as $(a,m)$ where $a\in A_{0,0}$ and $m\in M(\alpha)$, the $t$-action
on $M_0$ 
is given as $t\,(a,m)=(\sigma(m),t\, m)$, where $\sigma:M(\alpha)\to A_{0,0}$ is a 
surjective $\Lambda_0$-module homomorphism. Since for all such $\sigma$ we have
$\sigma(t\, m) = 0 = t^2\, m$, it follows that $t^2\, (a,m)=(\sigma(t\,m),t^2\,m)=(0,0)$ for all
$a\in A_{0,0}$ and $m\in M(\alpha)$. 
But this is a contradiction to $t^2M_0\cong A_{0,0}$. Thus $\psi$ is a $k$-algebra isomorphism and
$R(\Lambda_0,A_{0,0})\cong k[[t]]/(t^2)$.

Let $\underline{a}=\lambda\xi^{-1}\delta\rho^{-1}$ and consider $A_{0,2}=M(\underline{a})$. 
Then $A_{0,2}$ lies in  a non-split short exact sequence of $\Lambda_0$-modules
\begin{equation}
\label{eq:step41}
0\to A_{0,2}\to M(\underline{a}\beta\underline{a}) \to A_{0,2}\to 0.
\end{equation}
Therefore, $M(\underline{a}\beta\underline{a})$ defines a non-trivial lift of $A_{0,2}$ over
$k[[t]]/(t^2)$. Let $\{b_0,b_1,\ldots,b_4\}$ be a canonical $k$-basis of $A_{0,2}$ relative to the
representative $\underline{a}$, and let for each arrow $\zeta\in\{\alpha,\rho,\xi,\beta,\delta,\lambda\}$,
$X_\zeta$ be the $5\times 5$ matrix with entries in $k$ describing the action of $\zeta$ on $A_{0,2}$ 
with respect to $\{b_0,b_1,\ldots,b_4\}$  (see Definition \ref{def:strings}). 
Let $L_{0,2}$ be a free $k[[t]]$-module of rank $5$, and let
$\{B_0,B_1,\ldots,B_4\}$ be a basis of $L_{0,2}$ over $k[[t]]$. Viewing $k$ as a subalgebra of
$k[[t]]$, define a $\Lambda_0$-module structure on $L_{0,2}$
by letting $\zeta\in \{\alpha,\rho,\xi,\delta,\lambda\}$ act as the matrix $X_\zeta$ and
$\beta$ as the matrix $X_\beta + t E_{4,0}$ with respect to the basis $\{B_0,B_1,\ldots,B_4\}$.
Here $E_{4,0}$ is the elementary $5\times 5$ matrix which sends $B_0$ to $B_4$ and
all other basis elements to $0$. Then $L_{0,2}$ is a $k[[t]]\otimes_k\Lambda_0$-module
which is free over $k[[t]]$ and $L_{0,2}/tL_{0,2}\cong A_{0,2}$. Hence $L_{0,2}$ defines a
lift of $A_{0,2}$ over $k[[t]]$, and there exists a unique $k$-algebra homomorphism
$\varphi:R(\Lambda_0,A_{0,2})\to k[[t]]$ in $\hat{\mathcal{C}}$ corresponding to the deformation
defined by $L_{0,2}$. Because
$L_{0,2}/t^2L_{0,2}\cong M(\underline{a}\beta\underline{a})$ as $\Lambda_0$-modules, 
$L_{0,2}/t^2L_{0,2}$ defines a non-trivial lift of $A_{0,2}$ over $k[[t]]/(t^2)$. Thus 
$\varphi$ is surjective. Since 
$R(\Lambda_0,A_{0,2})$ is a quotient of $k[[t]]$,
this implies that $\varphi$ is an isomorphism, and hence $R(\Lambda_0,A_{0,2})\cong k[[t]]$.
This proves Proposition \ref{prop:simples}.
\end{proof}

\begin{prop}
\label{prop:stringtype0}
Let $S$ be the standard representative of a string of type 
$\beta\alpha^{-1}\lambda$ satisfying condition $(+)$ from 
Lemma $\ref{lem:needthis}$. Then 
${}_hS=\lambda\delta\rho^{-1}S$ and ${}_{hh}S=\delta\beta\alpha^{-1}({}_hS)$.
For $i\in\{0,1,2\}$, let $\mathfrak{C}_{S,i}$ be the component of the stable
Auslander-Reiten quiver of $\Lambda_0$ containing $M(\theta^i(S))$.
The component $\mathfrak{C}_{S,i}$ is stable under $\Omega$, and
the modules in $\mathfrak{C}_{S,i}$ whose stable endomorphism rings are isomorphic to $k$
are precisely the modules in the $\Omega$-orbits of the modules
$M(\theta^i(S))$, $M(\theta^i({}_hS))$ and $M(\theta^i({}_{hh}S))$.
Their universal deformation rings are
$$R(\Lambda_0, M(\theta^i(S))) \cong k[[t]]/(t^2),\qquad R(\Lambda_0, M(\theta^i({}_hS))) \cong k,
\qquad R(\Lambda_0, M(\theta^i({}_{hh}S))) \cong k[[t]].$$
\end{prop}

\begin{proof}
It suffices to consider the case when $i=0$.
Using hooks and cohooks (see Definition \ref{def:arcomps}), we see that $\mathfrak{C}_{S,0}$ is
stable under $\Omega$ and that all $\Lambda_0$-modules
in $\mathfrak{C}_{S,0}$ lie in the $\Omega$-orbit of either 
\begin{eqnarray*}
V_{n,0}&=&M\left(S(\delta\rho^{-1}\beta\alpha^{-1}\lambda\xi^{-1})^n\delta\rho^{-1}\right),\mbox{ or}\\
V_{n,1}&=&M\left(S(\delta\rho^{-1}\beta\alpha^{-1}\lambda\xi^{-1})^n\delta\rho^{-1}\beta\alpha^{-1}\right),\mbox{ or}\\
V_{n,2}&=&M\left(S(\delta\rho^{-1}\beta\alpha^{-1}\lambda\xi^{-1})^{n+1}\right)
\end{eqnarray*}
for some $n\ge 0$. Since $S=S^-$ when $S$ satisfies condition $(+)$, we have 
$M(S)\cong\Omega(V_{0,0})$, $M({}_hS)\cong\Omega(V_{0,1})$ and $M({}_{hh}S)\cong\Omega(V_{0,2})$.
By Lemma \ref{lem:needthis}, the stable endomorphism ring of $M(S)$ is isomorphic to $k$. 
Using Remark \ref{rem:stringhoms} together with similar arguments as in the proof of 
Lemma \ref{lem:needthis}, it follows that the stable endomorphism ring of $V_{0,j}$ is isomorphic to $k$ 
for $j\in\{0,1,2\}$. Similarly, we analyze $\mathrm{Ext}^1_{\Lambda_0}(V_{0,j},V_{0,j})$ to see that
it is isomorphic to $k$ for $j\in\{0,2\}$ and zero for $j=1$. 
On the other hand, for $n\ge 1$,  $V_{n,j}$ has a non-zero endomorphism which 
factors through $M(1_1)$ and which does not factor through a projective $\Lambda_0$-module. 

The universal deformation rings of $V_{0,0}$, $V_{0,1}$ and $V_{0,2}$ are determined in a similar
way as in the proof of Proposition \ref{prop:simples}. It follows that
$R(\Lambda_0,V_{0,0})\cong k[[t]]/(t^2)$, $R(\Lambda_0,V_{0,1})\cong k$ and
$R(\Lambda_0,V_{0,2})\cong k[[t]]$. 
Using Theorem \ref{thm:frobenius}(iv), this proves Proposition \ref{prop:stringtype0}.
\end{proof}

\begin{prop}
\label{prop:3tubes}
Let $\mathfrak{T}_1$ and $\mathfrak{T}_2$ be the two $3$-tubes of the stable Auslander-Reiten 
quiver of $\Lambda_0$. Then $\Omega(\mathfrak{T}_1) = \mathfrak{T}_2$. 
Define $T=\alpha^{-1}$. Then $T_h=\alpha^{-1}\lambda\xi^{-1}$ and 
$T_{hh}=\alpha^{-1}\lambda\xi^{-1}\delta\rho^{-1}$.
The modules in $\mathfrak{T}_1\cup \mathfrak{T}_2$ whose stable endomorphism rings are 
isomorphic to $k$ are precisely the modules in the $\Omega$-orbits of the modules
$M(T)$, $M(T_h)$ and $M(T_{hh})$.
Their universal deformation rings are
$$R(\Lambda_0, M(T)) \cong k \cong  R(\Lambda_0, M(T_h)),
\qquad R(\Lambda_0, M(T_{hh})) \cong k[[t]].$$
\end{prop}

\begin{proof}
Using 
 the description of the projective indecomposable
$\Lambda_0$-modules in (\ref{eq:proj}), we see that 
$\Omega(\mathfrak{T}_1) = \mathfrak{T}_2$. Using 
hooks and cohooks (see Definition \ref{def:arcomps}) and 
Remark \ref{rem:stringhoms}, it is straightforward to show that 
the only $\Lambda_0$-modules in $\mathfrak{T}_1\cup \mathfrak{T}_2$ whose
stable endomorphism rings are isomorphic to $k$ lie in the $\Omega$-orbit of
either $C_0=M(T)$, $C_1=M(T_h)$ or $C_2=M(T_{hh})$. Since
$\mathrm{Ext}^1_{\Lambda_0}(C_j,C_j)=0$ for $j\in\{0,1\}$, we have 
$R(\Lambda_0,C_j)\cong k$ for $j\in\{0,1\}$. Since $\mathrm{Ext}^1_{\Lambda_0}(C_2,C_2)
\cong k$, it follows that $R(\Lambda_0,C_2)$ is a quotient of $k[[t]]$.
Using similar arguments as in the proof of Proposition \ref{prop:simples}, we obtain
that $R(\Lambda_0,C_2)\cong k[[t]]$, which proves Proposition \ref{prop:3tubes}.
\end{proof}


\subsection{Stable endomorphism rings and universal deformation rings of band modules for $\Lambda_0$}
\label{ss:udrband}

In this section, we describe all band modules for $\Lambda_0$ whose stable endomorphism rings are isomorphic 
to $k$ and determine their universal deformation rings.

It follows from \cite{krau} that if $B$ is a band, $\mu\in k^*$ and $n\ge 2$ is an  
integer, then $\underline{\mathrm{End}}_{\Lambda_0}(M(B,\mu,n))$ has $k$-dimension
at least $2$. Hence we can concentrate on the band modules $M_{B,\mu}=M(B,\mu,1)$.

For the description of all the band modules whose stable endomorphism rings are
isomorphic to $k$, we need to look at special bands, which we call bands of type $0$
and which are defined using the words of type $0$ and their rotations of type $0$
introduced in Definition \ref{def:important}.

\begin{dfn}
\label{def:standardband}
\begin{enumerate}
\item[(i)]
We say a word $B$ representing a band is a  \emph{top-socle representative} if 
$$B=C_1C_2^{-1}C_3C_4^{-1}\cdots C_{2l-1}C_{2l}^{-1}$$
where $l\ge 1$, $C_{2i-1}\in\{\beta,\beta\lambda,\delta,\delta\beta,\lambda,\lambda\delta\}$
and $C_{2i}\in \{\alpha,\rho,\xi\}$ 
for $1\le i\le l$. The words $C_1,C_2,\ldots,C_{2l}$ are called the \emph{top-socle pieces} of $B$.
The positive integer $l$ is called the \emph{top-socle length} of $B$.
For $1\le j\le 2l$, $s(C_j)$ is called the \emph{top} of $C_j$ and $e(C_j)$ is called the \emph{socle} of $C_j$.
The collection of all the tops (resp. of all the socles) of the $C_{2i-1}$, $1\le i\le l$, where we
count multiplicities, is called 
the \emph{top} (resp. the \emph{socle}) of $B$.

\item[(ii)] We say a band is a \emph{band of type $0$} if it has a top-socle representative $B$ which is a 
word of type $0$, as defined in Definition \ref{def:important}(ii). 
In other words,
$$B=\underline{z}_0\, \underline{z}_1\,\cdots \,\underline{z}_{n-1}$$
where $n\ge 1$ and $\underline{z}_0,\ldots,\underline{z}_{n-1}\in\{\underline{x},\underline{y}\}$.
A \emph{standard representative of a band of type $0$} is a top-socle representative which is a word
of type $0$. This is unique up to rotations of type $0$, as defined in Definition \ref{def:important}(iv).

\item[(iii)] Let $B=\underline{z}_0\, \underline{z}_1\,\cdots \,\underline{z}_{n-1}$ be a standard 
representative of a band of type $0$. A \emph{rotated subpattern} $U$ of $B$ is a subpattern of 
$\rho_i^{(0)}(B)$ for some $0\le i\le n-1$, where $\rho_i^{(0)}$ is as in Definition 
\ref{def:important}(iv). Thus 
either $U=1_0$  or $U=\underline{z}_i\,\underline{z}_{i+1}\, \cdots \, \underline{z}_j$ for
some $0\le i\le  j\le n-1$ or 
$U=\underline{z}_i\,\underline{z}_{i+1} \cdots \,\underline{z}_{n-1}\underline{z}_0\,\cdots\, \underline{z}_j$
for some $0\le j<  i\le n-1$.
The \emph{inverse pattern} $U^-$ of $U$ is given
as in Definition \ref{def:important}(iii).

\item[(iv)] Let $B=\underline{z}_0\, \underline{z}_1\,\cdots \,\underline{z}_{n-1}$ be a standard 
representative of a band of type $0$, and let $\omega\in\{0,1,\ldots,n-1\}$. 
We say $B$ \emph{allows a wrap-around at $\omega$} if 
$\underline{z}_i=\underline{z}_{\omega-i}$ for all $i$, where the indices are taken modulo $n$. 
\end{enumerate}
\end{dfn}

\begin{rem}
\label{rem:maybe}
Suppose $B=\underline{z}_0\, \underline{z}_1\,\cdots \,\underline{z}_{n-1}$ is a standard 
representative of a band of type $0$.
\begin{enumerate}
\item[(i)] Since $B$ is not a power of a smaller word, $B$ allows at most one wrap-around.
\item[(ii)] Suppose $B$ allows a wrap-around at $\omega$. Let $\mu,\tilde{\mu}\in k^*$, and let
$M_{B,\mu}=M(B,\mu,1)$ and $M_{B,\tilde{\mu}}=M(B,\tilde{\mu},1)$.
We obtain the following sequence of canonical homomorphisms from $M_{B,\mu}$ to
$M_{B,\tilde{\mu}}$, as defined in Remark \ref{rem:bandendohelp}, associated to the wrap-around 
at $\omega$.

Let $\{b_0,b_1,\ldots,b_{5n-1}\}$ (resp. 
$\{\tilde{b}_0,\tilde{b}_1,\ldots,\tilde{b}_{5n-1}\}$) 
be a canonical $k$-basis for $M_{B,\mu}$ (resp. $M_{B,\tilde{\mu}}$)
relative to the representative $B$ (see Definition \ref{def:bands}). 
Let $0\le j\le n-1$. Define $\nu_j=1$ and $\epsilon_j=2$ if $\underline{z}_j=\underline{x}$ 
(resp. $\nu_j=2$ and $\epsilon_j=1$ if $\underline{z}_j=\underline{y}$). 
Then $\xi_{j,1}$ (resp. $\xi_{j,2}$) with $\xi_{j,1}(b_{5j+\nu_j})=\tilde{b}_{5(\omega-j)+\nu_{j}+1}$
and $\xi_{j,1}(b_i)=0$ for all other $i$ (resp. $\xi_{j,2}(b_{5j+4})=\tilde{b}_{5(\omega-j)}$ and 
$\xi_{j,2}(b_i)=0$ for all other $i$) defines a canonical homomorphism from
$M_{B,\mu}$ to $M_{B,\tilde{\mu}}$ of string type $1_{\epsilon_j}$ (resp. $1_0$).

Considering the projective indecomposable $\Lambda_0$-modules,
it follows that $\xi_{j,1}+\xi_{j,2}$ for $j\in\{0,1,\ldots,n-1\}-\{\omega\}$ and
$\xi_{\omega,1}+\tilde{\mu}\xi_{\omega,2}$ all factor through projective $\Lambda_0$-modules.
Also, $\xi_{j,2}+\xi_{j+1,1}$ for $j\in\{0,1,\ldots,n-2\}$ and $\xi_{n-1,2}+\mu\xi_{0,1}$
all factor through projective $\Lambda_0$-modules.
We conclude that every homomorphism in $\{\xi_{j,1},\xi_{j,2}\;|\;0\le j\le n-1\}$ factors
through a projective $\Lambda_0$-module if and only if  $\mu\tilde{\mu}\neq 1$.
\end{enumerate}
\end{rem}

\begin{lemma}
\label{lem:needthisforbands}
Let $B=\underline{z}_0\underline{z}_1\cdots  \underline{z}_{n-1}$
be a standard representative of a band $B$ of type $0$ as in Definition $\ref{def:standardband}$,
let $\mu\in k^*$ and let $M_{B,\mu}=M(B,\mu,1)$. The stable endomorphism ring of $M_{B,\mu}$, 
$\underline{\mathrm{End}}_{\Lambda_0}(M_{B,\mu})$, is isomorphic to $k$ if and only if
$\mu\neq \pm 1$ and $B$ has the following property:
\begin{enumerate}
\item[(++)] 
If there are rotated subpatterns $U,V$ of $B$ such that 
$B\sim_{r,(0)}\underline{x}\,U\,\underline{x}\,V$, 
then for all rotated subpatterns $W$ of $B$, we have 
$B\not\sim_{r,(0)} \underline{y}\,U\,\underline{y}\, W$. 
\end{enumerate}
Here $\sim_{r,(0)}$ is the equivalence relation introduced in Definition $\ref{def:important}(iv)$.
\end{lemma}

\begin{proof}
Since for $B\in\{\underline{x},\underline{y}\}$,
$\underline{\mathrm{End}}_{\Lambda_0}(M_{B,\mu})\cong k$ if and only if 
$\mu\neq\pm 1$ and since $\underline{x}$ and $\underline{y}$ have property $(++)$, 
Lemma \ref{lem:needthisforbands}
follows for $n=1$. For the remainder of the proof, let 
$B=\underline{z}_0\underline{z}_1\cdots  \underline{z}_{n-1}$ be a standard representative of
a band of type $0$ with $n\ge 2$, and let $\mu\in k^*$. 
We analyze the different possibilities for $M_{B,\mu}$ such that
$\underline{\mathrm{End}}_{\Lambda_0}(M_{B,\mu})$ has $k$-dimension at least $2$.

Suppose first that there exists a string $T$ of length at least $2$ such that $M_{B,\mu}$ has a canonical
endomorphism $\alpha_T$ of string type $T$, as described in Remark \ref{rem:bandendohelp}. 
In particular,  $\underline{\mathrm{End}}_{\Lambda_0}(M_{B,\mu})$ has $k$-dimension at least $2$. 
It follows from the definition of bands of type $0$ that
$$T\sim_s\beta\alpha^{-1}U\lambda$$
for some rotated subpattern $U$ of $B$
and that there exist rotated subpatterns $V,W$ of $B$
such that
\begin{itemize}
\item[$(\ddagger)$] $B\sim_{r,(0)}\underline{x}\,U\,\underline{x}\,V$ and
$B\sim_{r,(0)}\underline{y}\, U \, \underline{y}\,W$. 
\end{itemize}
On the other hand, if $B$ satisfies $(\ddagger)$, then $M_{B,\mu}$ has a canonical
endomorphism factoring through a string module with at least three composition factors.

Next let $T$ be a string of length $0$ or $1$ such that $M_{B,\mu}$ has a canonical
endomorphism $\alpha_T$ of string type $T$. Then either $T=\delta$ or $T=1_u$ for some
$u\in\{0,1,2\}$. 

Analyzing these cases further, it follows that if $T=\delta$ then
$\alpha_T$ does not factor through a projective 
$\Lambda_0$-module only if $B$ satisfies $(\ddagger)$
for certain rotated subpatterns $U,V,W$ of $B$.

Next let $T=1_u$ for some $u\in\{0,1,2\}$. Then $\alpha_T$ does not factor through a projective 
$\Lambda_0$-module only if either (a) $B$ satisfies $(\ddagger)$ for 
certain $U,V,W$, or (b) $B$ has property $(++)$ and $\mu=\pm 1$.
To explain (b), we first note that if $B$ has property $(++)$, then $B$ allows a wrap-around at 
$\omega$ for a unique $\omega\in\{0,1,\ldots,n-1\}$ (see  Proposition \ref{prop:bandstype0}).
Using Remark \ref{rem:maybe}(ii), it follows that
if $B$ satisfies $(++)$ then $\alpha_T$ does not factor through a projective $\Lambda_0$-module
if and only if there exist $0\le j\le n-1$ and $s\in\{1,2\}$ such that $\alpha_T=\xi_{j,s}$
and $\mu=\pm 1$.
This completes the proof of Lemma \ref{lem:needthisforbands}.
\end{proof}

\begin{dfn}
\label{def:pandq}
Define
\begin{equation}
\label{eq:list1}
\underline{p}=\lambda\xi^{-1}\delta\rho^{-1}\beta\alpha^{-1}
\quad\mbox{and}\quad
\underline{q}=\lambda\delta\rho^{-1}\beta\lambda\xi^{-1}\delta\beta\alpha^{-1}.
\end{equation}
\end{dfn}

\begin{thm}
\label{thm:stablendbands}
Let $M$ be a band module for $\Lambda_0$. Then the stable  endomorphism ring of $M$ 
is isomorphic to $k$ if and only if either
\begin{enumerate}
\item[(i)] there exist $\mu\in k^*$ and $B\in\{\underline{p},\underline{q}\}$ such that
$M\cong M_{B,\mu}$, or
\item[(ii)] there exist $i\in\{0,1,2\}$, $\mu\in k^*$ with $\mu\neq \pm 1$, and
a standard representative $B$ of a band of type $0$
satisfying condition $(++)$ from Lemma $\ref{lem:needthisforbands}$ such that
$M\cong M_{\theta^i(B),\mu}$. 
\end{enumerate}
\end{thm}

\begin{proof}
Let $C$ be a top-socle representative of a band whose top-socle length is $l\ge 1$,
and let $\mu\in k^*$ be such that $M\cong M_{C,\mu}$. 
As in Definition \ref{def:standardband}(i), $C$ has the form 
$$C=D_1D_2^{-1}D_3D_4^{-1}\cdots D_{2l-1}D_{2l}^{-1}$$
where $D_{2j-1}\in\{\beta,\beta\lambda,\delta,\delta\beta,\lambda,\lambda\delta\}$
and $D_{2j}\in\{\alpha,\rho,\xi\}$.
In particular, $l\ge 2$.

If all $D_{2j-1}$ have the same length, then either $C\sim_r \underline{p}$ or 
$C\sim_r \underline{q}$. 
In this case the stable endomorphism 
ring of $M_{C,\mu}$ is isomorphic to $k$ for all $\mu$.

For the remainder of the proof suppose that not all $D_{2j-1}$ have the same length. 
Then there exists $i_0\in\{0,1,2\}$ such that $C$ has a subword of the form 
$\theta^{i_0}(\underline{x})$ or $\theta^{i_0}(\underline{y})$. This means that we can write
$$\tilde{C} \stackrel{\mathrm{def}}{=} \theta^{-i_0}(C) = Z_0 Z_1\cdots Z_{m-1}$$
where for all $0\le j\le m-1$, 
$Z_j$ 
is a sequence of top-socle pieces
of $\tilde{C}$ such that $Z_j=\lambda\cdots \alpha^{-1}$ and $1_0$ occurs exactly 
once in the top of $Z_j$. 
Hence $Z_j\in\{\underline{p}_r,\underline{q}_r,\underline{x}_r,\underline{y}_r\;|\;r\ge 0\}$,
where
$$\underline{p}_r=
\lambda\xi^{-1}\delta\rho^{-1}\left(\beta\lambda\xi^{-1}\delta\rho^{-1}\right)^r\beta\alpha^{-1},\quad
\underline{q}_r=\lambda\left(\delta\rho^{-1}\beta\lambda\xi^{-1}\right)^{r+1}\delta\beta\alpha^{-1},$$
$$\underline{x}_r=\lambda\xi^{-1}\left(\delta\rho^{-1}\beta\lambda\xi^{-1}\right)^r\delta\beta\alpha^{-1},
\quad
\underline{y}_r=\lambda\delta\rho^{-1}\left(\beta\lambda\xi^{-1}\delta\rho^{-1}\right)^r\beta\alpha^{-1}.$$
In particular, 
$\underline{p}_0=\underline{p}$, $\underline{q}_0=\underline{q}$, $\underline{x}_0=\underline{x}$,
$\underline{y}_0=\underline{y}$.
Using that at least one of the $Z_j$ is $\underline{x}_0$ or $\underline{y}_0$, we obtain that
the stable endomorphism ring of $M_{\tilde{C},\mu}$, and hence of $M_{C,\mu}$, has $k$-dimension
at least 2 for all $\mu\in k^*$ except in the following cases: 
\begin{enumerate}
\item[(a)] $\tilde{C}$ only contains $\underline{x}_0$ and $\underline{y}_0$; or
\item[(b)] $\tilde{C}$ only contains $\underline{x}_0$ and $\underline{x}_1$; or
\item[(c)] $\tilde{C}$ only contains $\underline{y}_0$ and $\underline{y}_1$.
\end{enumerate}
Suppose that $\tilde{C}$ is as in one of the cases (a) -- (c). Then there exists $i_1\in\{0,1,2\}$
such that a  rotation of  $\theta^{i_1}(\tilde{C})$ is a standard representative of a band of type $0$. 
Thus  there exist $i\in\{0,1,2\}$ and a standard representative $B$ of a band of 
type $0$ such that  $M_{C,\mu}\cong M_{\theta^i(B),\mu}$.
By Lemma \ref{lem:needthisforbands}, it follows
that the stable endomorphism ring of $M_{C,\mu}$ is isomorphic to $k$ if and only if
$\mu\neq \pm 1$ and $B$ satisfies condition $(++)$.
This completes the proof of Theorem \ref{thm:stablendbands}.
\end{proof}

Because of the prominent role played by standard representatives of bands $B$ of type 
$0$ in Theorem \ref{thm:stablendbands}, the question arises 
which such $B$ satisfy condition $(++)$ from Lemma \ref{lem:needthisforbands}. We will give a precise
answer to this question in Proposition \ref{prop:bandstype0}.

We now look at the band modules for $\Lambda_0$ singled out by Theorem
\ref{thm:stablendbands}.

\begin{prop}
\label{prop:udrbands}
Let $B$ be a top-socle representative of a band and let $\mu\in k^*$. 
\begin{enumerate}
\item[(i)] 
If $\mu\in k^*$ and $\{B,B'\}=\{\underline{p},\underline{q}\}$, then 
$\Omega(M_{B,\mu})\cong M_{B',-\mu^{-1}}$.
\item[(ii)] If $\mu\in k^*-\{ \pm 1\}$ and there exists $i\in\{0,1,2\}$ such that $\theta^{-i}(B)$ is 
a standard representative of a band of type $0$ satisfying condition $(++)$ from Lemma 
$\ref{lem:needthisforbands}$, then  $\Omega(M_{B,\mu})\cong M_{B,\mu^{-1}}$.
\end{enumerate}
In each of the above cases, 
the universal deformation ring is $R(\Lambda_0,M_{B,\mu})\cong k[[t]]$.
\end{prop}

\begin{proof}
Using that $\Lambda_0=P_2\oplus P_1\oplus P_0$ is a projective $\Lambda_0$-module cover of 
$M_{\underline{p},\mu}$ and analyzing the kernel $\Omega(M_{\underline{p},\mu})$ of the natural 
surjection $\Lambda_0 \to M_{\underline{p},\mu}$, it follows that 
$\Omega(M_{\underline{p},\mu})\cong M_{\underline{q},-\mu^{-1}}$.
Since $\Omega^2(M_{\underline{p},\mu})\cong M_{\underline{p},\mu}$, we also have 
$\Omega(M_{\underline{q},\mu})\cong M_{\underline{p},-\mu^{-1}}$.
Using Remark \ref{rem:bandendohelp}, we see that
the vector space $\mathrm{Hom}_{\Lambda_0}(M_{\underline{q},-\mu^{-1}},M_{\underline{p},\mu})$
is $3$-dimensional over $k$, but the subspace of homomorphisms factoring through projective
$\Lambda_0$-modules has $k$-dimension $2$. Hence 
$$\mathrm{Ext}^1_{\Lambda_0}(M_{\underline{p},\mu},M_{\underline{p},\mu})\cong
\underline{\mathrm{Hom}}_{\Lambda_0}(\Omega(M_{\underline{p},\mu}),M_{\underline{p},\mu})\cong
\underline{\mathrm{Hom}}_{\Lambda_0}(M_{\underline{q},-\mu^{-1}},M_{\underline{p},\mu})\cong k.$$
This implies that $R(\Lambda_0,M_{\underline{p},\mu})$ is a quotient of $k[[t]]$.
Using similar arguments as in the proof of Proposition \ref{prop:simples}, we obtain
that $R(\Lambda_0,M_{\underline{p},\mu})\cong k[[t]]$.
By Theorem \ref{thm:frobenius}(iv), we also get $R(\Lambda_0,M_{\underline{q},\mu})\cong k[[t]]$,
which proves part (i).

Suppose now that $\mu$, $i$ and $B$ are as in part (ii). Without loss of generality, we may assume
that $i=0$, which means that $B$ is a standard representative of a band of type $0$ satisfying 
condition $(++)$. 
Then $B$ allows a unique wrap-around (see Proposition \ref{prop:bandstype0}), 
which implies $B\sim_{r,(0)} B^-$. 
Since the top-socle length of $B$ is even, it follows that
$\Omega(M_{B,\mu}) \cong M_{B,\mu^{-1}}$.
Using Remark \ref{rem:bandendohelp}, we see that the $k$-dimension of 
the vector space $\mathrm{Hom}_{\Lambda_0}(M_{B,\mu^{-1}},M_{B,\mu})$ is one less than the
$k$-dimension of  $\mathrm{End}_{\Lambda_0}(M_{B,\mu})$,
since there is no identity homomorphism in the former space. 
Arguing in a similar way as in the proof 
that the stable endomorphism ring of $M_{B,\mu}$ is isomorphic to $k$, we see that all the 
canonical homomorphisms $M_{B,\mu^{-1}}\to M_{B,\mu}$ 
factor through projective 
$\Lambda_0$-modules except possibly the homomorphisms  $\xi_{j,1},\xi_{j,2}$, 
$0\le j\le n-1$,  from  Remark \ref{rem:maybe}(ii). Since $\mu^{-1}\mu=1$, 
it follows that
$$\mathrm{Ext}^1_{\Lambda_0}(M_{B,\mu},M_{B,\mu})\cong
\underline{\mathrm{Hom}}_{\Lambda_0}(\Omega(M_{B,\mu}),M_{B,\mu})\cong
\underline{\mathrm{Hom}}_{\Lambda_0}(M_{B,\mu^{-1}},M_{B,\mu})\cong k.$$
This implies that $R(\Lambda_0,M_{B,\mu})$ is a quotient of $k[[t]]$.
Using similar arguments as in the proof of Proposition \ref{prop:simples}, we obtain
that $R(\Lambda_0,M_{B,\mu})\cong k[[t]]$.
This completes the proof of Proposition \ref{prop:udrbands}
\end{proof}


\section{Appendix: Strings and bands for $\Lambda_0$ satisfying $(+)$ and $(++)$ from Section 
$\ref{s:particular}$}
\label{s:plus}
\setcounter{equation}{0}

We assume the notation from \S \ref{s:particular}. In particular,  $k$ is an algebraically closed field of 
arbitrary characteristic and $\Lambda_0$ is the basic $k$-algebra $\Lambda_0=kQ/I$ where 
$Q$ and $I$ are as in Figure \ref{fig:algebra}. 
Because of the prominent role played by standard representatives of strings $S$ of type 
$\beta\alpha^{-1}\lambda$ satisfying condition $(+)$ in Theorem \ref{thm:stablend},
respectively by standard representatives of bands $B$ of type $0$ satisfying condition $(++)$
in Theorem \ref{thm:stablendbands}, the question arises if one can give an explicit description
of such $S$ and $B$. In this appendix,
we give a precise answer to this question in Propositions \ref{prop:string0+} and
\ref{prop:bandstype0}.

Considering words of type $0$ of small type-$0$-length suggests that  to answer this question
one needs to use an inductive process. The key ingredients in this inductive process are described 
in the following definition. We assume Definition \ref{def:important}.

\begin{dfn}
\label{def:important2}
\begin{enumerate}
\item[(i)] Let $\ell$ be a positive integer, 
let $(i_1,\ldots, i_{\ell-1})$ be an $(\ell-1)$-tuple with $i_j\in\{0,1\}\mod 2$ for $1\le j\le \ell-1$,
and let $(a_1,\ldots,a_{\ell-1})$ be an $(\ell-1)$-tuple of positive integers. We define
$$N_0=\underline{x},\qquad N_1=\underline{y};$$
and for $1\le j\le \ell-1$,
\begin{eqnarray*}
N_{i_1,\ldots,i_j,0}(a_1,\ldots,a_j) &=& N_{i_1,\ldots,i_j}(a_1,\ldots,a_{j-1})^{a_j}\,
N_{i_1,\ldots,i_{j-1},(i_j)+1}(a_1,\ldots,a_{j-1}),\\
N_{i_1,\ldots,i_j,1}(a_1,\ldots,a_j) &=& N_{i_1,\ldots,i_j}(a_1,\ldots,a_{j-1})^{a_j+1}\,
N_{i_1,\ldots,i_{j-1},(i_j)+1}(a_1,\ldots,a_{j-1}).
\end{eqnarray*}

\item[(ii)] Let $\ell$ be a positive integer, 
let $(i_1,\ldots, i_{\ell-1})$ be an $(\ell-1)$-tuple with $i_j\in\{0,1\}\mod 2$ for $1\le j\le \ell-1$,
and let $(a_1,\ldots,a_{\ell-1})$ be an $(\ell-1)$-tuple of positive integers. 
Define 
\begin{eqnarray*}
\underline{x}^{(\ell)}&=&N_{i_1,\ldots,i_{\ell-1},0}(a_1,\ldots,a_{\ell-1}),\\
\underline{y}^{(\ell)} &=&  N_{i_1,\ldots,i_{\ell-1},1}(a_1,\ldots,a_{\ell-1}).
\end{eqnarray*}
We say a word $Z$ is a \emph{word of type $0$ on level $\ell$} if $Z=1_0$ or
$$Z=\underline{z}^{(\ell)}_0\cdots \underline{z}^{(\ell)}_{n-1}$$
where $n\ge 1$ and 
$\underline{z}^{(\ell)}_0,\ldots,\underline{z}^{(\ell)}_{n-1}\in\{\underline{x}^{(\ell)},\underline{y}^{(\ell)}\}$.
We call $n$ the \emph{type-$0$-length on level $\ell$} of $Z$, where $n=0$
corresponds to $Z=1_0$. Note that for $\ell=1$, this is the
original definition of a word of type $0$ from Definition \ref{def:important}.

If we want to emphasize the tuples used to define 
$\underline{x}^{(\ell)}$ and $\underline{y}^{(\ell)}$, we say $Z$ is a word of type $0$ on level $\ell$
relative to the $(\ell-1)$-tuples $(i_1,\ldots,i_{\ell-1})$ and $(a_1,\ldots,a_{\ell-1})$.

\item[(iii)] Let $Z=\underline{z}^{(\ell)}_0\cdots \underline{z}^{(\ell)}_{n-1}$ be a word of type $0$ 
on level $\ell$. A \emph{subpattern} $U$ of $Z$ \emph{on level $\ell$} is a subword
of $Z$ such that either $U=1_0$  or 
$U=\underline{z}^{(\ell)}_i \underline{z}^{(\ell)}_{i+1}\cdots  \underline{z}^{(\ell)}_j$ 
for some $0\le i\le  j\le n-1$.

\item[(iv)] Let $Z=\underline{z}^{(\ell)}_0\cdots \underline{z}^{(\ell)}_{n-1}$ be a word of type $0$ 
on level $\ell$ with $n\ge 1$. For $0\le i\le n-1$, we define
the \emph{$i^{\mathrm{th}}$ rotation of $Z$ of type $0$ on level $\ell$} to be
$$\rho_i^{(0,\ell)}(Z) = \underline{z}^{(\ell)}_i\underline{z}^{(\ell)}_{i+1}\cdots \underline{z}^{(\ell)}_{n-1}
\underline{z}^{(\ell)}_0\cdots \underline{z}^{(\ell)}_{i-1}.$$
Define $\sim_{r,(0,\ell)}$ to be the equivalence relation on all words of type $0$ on level $\ell$
of type-0-length on level $\ell$ at least 1 such that $Z\sim_{r,(0,\ell)} Z'$ if and only if 
$Z=\rho_i^{(0,\ell)}(Z')$ for some $i$.
Note that for $\ell=1$, the equivalence relation $\sim_{r,(0,\ell)}$ coincides with the
equivalence relation $\sim_{r,(0)}$ from Definition \ref{def:important}.

\item[(v)] Let $\ell$ be a positive integer, 
let $(i_1,\ldots, i_{\ell-1})$ be an $(\ell-1)$-tuple with $i_j\in\{0,1\}\mod 2$ for $1\le j\le \ell-1$,
and let $(a_1,\ldots,a_{\ell-1})$ be an $(\ell-1)$-tuple of positive integers. 
Define $\underline{x}^{(\ell)}$ and $\underline{y}^{(\ell)}$ as in part (ii) and define
\begin{eqnarray*}
\underline{b}^{(\ell)} &=& 
	\beta\alpha^{-1}N_{i_1,0}(a_1)\,N_{i_1,i_2,0}(a_1,a_2)\,\cdots \,
	N_{i_1,\ldots,i_{\ell-1},0}(a_1,\ldots,a_{\ell-1}),\\
\underline{l}^{(\ell)} &=& 
	N_{i_1,\ldots,i_{\ell-1}}(a_1,\ldots,a_{\ell-2})^{a_{\ell-1}}\,\cdots\,
	N_{i_1,i_2}(a_1)^{a_2}\,(N_{i_1})^{a_1}\,\lambda.
\end{eqnarray*}
We say a word $S$ is the \emph{standard representative of a string of type 
$\beta\alpha^{-1}\lambda$ on level $\ell$} if
$$S=\underline{b}^{(\ell)} Z\,\underline{l}^{(\ell)}$$
where $Z$ is a word of type $0$ on level $\ell$. 
Note that for $\ell=1$, this is the
original definition of the standard representative of a string of type 
$\beta\alpha^{-1}\lambda$ from Definition \ref{def:important1}.

If we want to emphasize the tuples used to define $\underline{x}^{(\ell)}$, $\underline{y}^{(\ell)}$,
$\underline{b}^{(\ell)}$ and $\underline{l}^{(\ell)}$, we say $S$ is the standard
representative of a string of type $\beta\alpha^{-1}\lambda$ on level $\ell$
relative to the $(\ell-1)$-tuples $(i_1,\ldots,i_{\ell-1})$ and $(a_1,\ldots,a_{\ell-1})$.

A \emph{subpattern} $U$ of $S$ \emph{on level $\ell$} is a 
subpattern of $Z$ on level $\ell$, as defined in part (iii). 

\item[(vi)] Let $\ell$ be a positive integer, 
let $(i_1,\ldots, i_{\ell-1})$ be an $(\ell-1)$-tuple with $i_j\in\{0,1\}\mod 2$ for $1\le j\le \ell-1$,
and let $(a_1,\ldots,a_{\ell-1})$ be an $(\ell-1)$-tuple of positive integers. 
Define $\underline{x}^{(\ell)}$ and $\underline{y}^{(\ell)}$ as in part (ii).
We say a band is a \emph{band of type $0$ on level $\ell$}
if it has a top-socle representative $B$ which is a word of type $0$ on level $\ell$. 
In other words,
$$B=\underline{z}^{(\ell)}_0\cdots \underline{z}^{(\ell)}_{n-1}$$
where $n\ge 1$ and 
$\underline{z}^{(\ell)}_0,\ldots,\underline{z}^{(\ell)}_{n-1}\in\{\underline{x}^{(\ell)},\underline{y}^{(\ell)}\}$.
A \emph{standard representative of a band of type $0$ on level $\ell$} is a top-socle representative 
which is a word of type $0$ on level $\ell$. This is unique up to rotations of type $0$ on level $\ell$,
as defined in part (iv).
Note that for $\ell=1$, this is the original definition of a band of type $0$ and a standard
representative of a band of type $0$ from Definition 
\ref{def:standardband}.

If we want to emphasize the tuples used to define 
$\underline{x}^{(\ell)}$ and $\underline{y}^{(\ell)}$, we say $B$ is a standard
representative of a band of type $0$ on level $\ell$
relative to the $(\ell-1)$-tuples $(i_1,\ldots,i_{\ell-1})$ and $(a_1,\ldots,a_{\ell-1})$.

\item[(vii)] Let $B=\underline{z}^{(\ell)}_0\, \underline{z}^{(\ell)}_1\,\cdots \,\underline{z}^{(\ell)}_{n-1}$ be a 
standard representative of a band of type $0$ on level $\ell$. A \emph{rotated subpattern} $U$ of $B$ 
\emph{on level $\ell$} is a subpattern of $\rho^{(0,\ell)}_i(B)$ for some $0\le i\le n-1$,
where $\rho^{(0,\ell)}_i$ is as in part (iv). Thus
either $U=1_0$  or $U=\underline{z}^{(\ell)}_i\,\underline{z}^{(\ell)}_{i+1}\, \cdots \, 
\underline{z}^{(\ell)}_j$  for some $0\le i\le  j\le n$ or
$U=\underline{z}^{(\ell)}_i\,\underline{z}^{(\ell)}_{i+1} \cdots \,\underline{z}^{(\ell)}_{n-1}
\underline{z}^{(\ell)}_0\,\cdots\, \underline{z}^{(\ell)}_j$ for some $0\le j<  i\le n$.

\item[(viii)] Let $B=\underline{z}^{(\ell)}_0\, \underline{z}^{(\ell)}_1\,\cdots \,\underline{z}^{(\ell)}_{n-1}$ 
be a standard representative of a band of type $0$ on level $\ell$, and let $\omega\in\{0,1,\ldots,n-1\}$. 
We say $B$ \emph{allows a wrap-around on level $\ell$ at $\omega$} if 
$\underline{z}^{(\ell)}_i=\underline{z}^{(\ell)}_{\omega-i}$ for all $i$, where the indices are taken modulo 
$n$. 
\end{enumerate}
\end{dfn}

\begin{rem}
\label{rem:easy}
\begin{enumerate}
\item[(i)]
Let $S$ be the standard representative of a string of type $\beta\alpha^{-1}\lambda$. Then
$S$ is the standard representative of a string of type $\beta\alpha^{-1}\lambda$ on level $1$.
There is a maximal positive integer $\ell$ such that $S$ is the standard representative of a string 
of type $\beta\alpha^{-1}\lambda$ on level $\ell$.

Conversely, if $S$ is the standard representative of a string of type $\beta\alpha^{-1}\lambda$ 
on level $\ell$ relative to the $(\ell-1)$-tuples $(i_1,\ldots,i_{\ell-1})$ and $(a_1,\ldots,a_{\ell-1})$,
then $S$ is the standard representative of a string of type $\beta\alpha^{-1}\lambda$ 
on level $j$ for all $1\le j\le \ell$ relative to the $(j-1)$-tuples $(i_1,\ldots,i_{j-1})$ and 
$(a_1,\ldots,a_{j-1})$.

\item[(ii)]
Let $B$ be a standard representative of a band of type $0$. Then
$B$ is a standard representative of a band of type $0$ on level $1$.
There is a maximal positive integer $\ell$ such that $B$ lies in the equivalence class under
$\sim_{r,(0)}$ of a standard representative of a band
of type $0$ on level $\ell$.

Conversely, if $B$ is a standard representative of a band of type $0$ 
on level $\ell$ relative to the $(\ell-1)$-tuples $(i_1,\ldots,i_{\ell-1})$ and $(a_1,\ldots,a_{\ell-1})$,
then $B$ is a standard representative of a band of type $0$ 
on level $j$ for all $1\le j\le \ell$ relative to the $(j-1)$-tuples $(i_1,\ldots,i_{j-1})$ and 
$(a_1,\ldots,a_{j-1})$.
\end{enumerate}
\end{rem}

\begin{prop}
\label{prop:string0+}
Let $S=\beta\alpha^{-1}\underline{z}_0 \cdots \underline{z}_{n-1}\,\lambda$
be the standard representative of a string of type 
$\beta\alpha^{-1}\lambda$ as in Definition $\ref{def:important1}$.
Then $S$ has property $(+)$ from Lemma $\ref{lem:needthis}$ if and only if
there exist a positive integer $\ell$ and a choice of $(\ell-1)$-tuples 
$(i_1,\ldots, i_{\ell-1})$ and $(a_1,\ldots,a_{\ell-1})$ as in Definition $\ref{def:important2}(v)$
such that
\begin{equation}
\label{eq:formstring0}
S=\underline{b}^{(\ell)}\,\underline{l}^{(\ell)} \;\mbox{ or }\; S=\underline{b}^{(\ell)}\,(\underline{x}^{(\ell)})^{a_\ell}\,\underline{l}^{(\ell)} \;\mbox{ or } \;
S=\underline{b}^{(\ell)}\,(\underline{y}^{(\ell)})^{a_\ell}\,\underline{l}^{(\ell)}
\end{equation}
for some positive integer $a_\ell$.
\end{prop}

\begin{proof}
Let $\ell$ be a positive integer, let $(i_1,\ldots, i_{\ell-1})$ and $(a_1,\ldots,a_{\ell-1})$ be
$(\ell-1)$-tuples as in Definition \ref{def:important2}(v), and let
$S$ be the standard representative of a string of type 
$\beta\alpha^{-1}\lambda$ on level $\ell$ relative to these $(\ell-1)$-tuples. 
As stated in Remark \ref{rem:easy}(i),
$S$ is then  also the standard representative of 
a string of type $\beta\alpha^{-1}\lambda$ on level $j$
for all $1\le j\le \ell$ relative to the $(j-1)$-tuples $(i_1,\ldots, i_{j-1})$ and $(a_1,\ldots,a_{j-1})$.
Let $j\in\{1,\ldots,\ell\}$. Then
$S=\underline{b}^{(j)}\,\underline{z}^{(j)}_0\cdots \underline{z}^{(j)}_{n_j-1}\,\underline{l}^{(j)}$,
where  $\underline{z}^{(j)}_i\in\{\underline{x}^{(j)},\underline{y}^{(j)}\}$ for all $i\in\{0,\ldots,n_j-1\}$. 
Let $\underline{c}^{(j)}=\underline{z}^{(j)}_0$ and let $\underline{d}^{(j)}$ be such that
$\{\underline{c}^{(j)},\underline{d}^{(j)}\}=\{\underline{x}^{(j)},\underline{y}^{(j)}\}$.
In particular, for $1\le j<\ell$, $\underline{x}^{(j+1)}=(\underline{c}^{(j)})^{a_j}\underline{d}^{(j)}$
and $\underline{y}^{(j+1)}=(\underline{c}^{(j)})^{a_j+1}\underline{d}^{(j)}$, and
$\underline{b}^{(j+1)}=\underline{b}^{(j)}\,(\underline{c}^{(j)})^{a_j}\underline{d}^{(j)}$ and 
$\underline{l}^{(j+1)}=(\underline{c}^{(j)})^{a_j}\underline{l}^{(j)}$.

For $1\le j\le \ell$, consider the following property:
\begin{enumerate}
\item[$(+)_j$] 
If there are subpatterns $U,A_1,A_2$ of $S$ on level $j$ such that
$$S=\underline{b}^{(j)} \,U\, \underline{c}^{(j)}A_2\,\underline{l}^{(j)} \;\mbox{ or } \;
S=\underline{b}^{(j)}A_1\,\underline{c}^{(j)}\,U\,\underline{c}^{(j)}A_2\,\underline{l}^{(j)} \;\mbox{ or }\;
S=\underline{b}^{(j)}A_1\,\underline{c}^{(j)}\,U\,\underline{l}^{(j)},$$ then 
for all subpatterns $B_1,B_2$ of $S$ on level $j$, we have
$$S\neq \underline{b}^{(j)}\, U\, \underline{d}^{(j)}B_2\,\underline{l}^{(j)}\;\mbox{ and }\;
S\neq \underline{b}^{(j)}B_1\,\underline{d}^{(j)}\,U\,\underline{d}^{(j)} B_2\,\underline{l}^{(j)} 
\;\mbox{ and }\; S\neq \underline{b}^{(j)}B_1\,\underline{d}^{(j)}\,U\,\underline{l}^{(j)}.$$
\end{enumerate}

\medskip

\noindent\textit{Claim $1$.} Let $S$ be the standard representative of a string of type 
$\beta\alpha^{-1}\lambda$ on level $\ell$ as in the 
first paragraph of the proof. 
Then $S$ has property $(+)$ from Lemma \ref{lem:needthis} if and only if 
$S$ has property $(+)_\ell$.

\medskip

\noindent
\textit{Proof of Claim $1$.}
Note that $(+)_1$ is the same as $(+)$. Moreover, $S$ does not have property $(+)_j$
for some $1\le j\le\ell$ if and only if there exist subpatterns $U,A_1,A_2,B_1,B_2$ of $S$ on level $j$ 
satisfying $(\ast)_j$ and $(\ast\ast)_j$, where
\begin{enumerate}
\item[$(\ast)_j$] $S=\underline{b}^{(j)} \,U\, \underline{c}^{(j)}A_2\,\underline{l}^{(j)}$ or 
$S=\underline{b}^{(j)}A_1\,\underline{c}^{(j)}\,U\,\underline{c}^{(j)}A_2\,\underline{l}^{(j)}$ or 
$S=\underline{b}^{(j)}A_1\,\underline{c}^{(j)}\,U\,\underline{l}^{(j)}$, 
\item[$(\ast\ast)_j$]  $S=\underline{b}^{(j)}\, U\, \underline{d}^{(j)}B_2\,\underline{l}^{(j)}$ or
$S=\underline{b}^{(j)}B_1\,\underline{d}^{(j)}\,U\,\underline{d}^{(j)} B_2\,\underline{l}^{(j)}$ or
$S=\underline{b}^{(j)}B_1\,\underline{d}^{(j)}\,U\,\underline{l}^{(j)}$. 
\end{enumerate}
Suppose that $S$ does not have property $(+_j)$, i.e. $S$ satisfies $(\ast)_j$ and $(\ast\ast)_j$.

If $j<\ell$, then $S$ is also the standard representative of a string of type 
$\beta\alpha^{-1}\lambda$ on level $(j+1)$. Therefore, it follows
that the subpattern $U$ in $(\ast)_j$ and $(\ast\ast)_j$ has the form
$U=(\underline{c}^{(j)})^{a_j}\underline{d}^{(j)}\,\tilde{U}\,(\underline{c}^{(j)})^{a_j}$ for some 
subpattern $\tilde{U}$ of $S$ on level $(j+1)$. 
Since 
$\underline{b}^{(j+1)}=\underline{b}^{(j)}\,(\underline{c}^{(j)})^{a_j}\underline{d}^{(j)}$,
$\underline{l}^{(j+1)}=(\underline{c}^{(j)})^{a_j}\underline{l}^{(j)}$,
$\underline{x}^{(j+1)}=(\underline{c}^{(j)})^{a_j}\underline{d}^{(j)}$
and $\underline{y}^{(j+1)}=(\underline{c}^{(j)})^{a_j+1}\underline{d}^{(j)}$, it follows that 
$S$ satisfies $(\ast)_{j+1}$ and $(\ast\ast)_{j+1}$ for appropriate subpatterns of $S$ on level $(j+1)$.

If $j>1$, then $S$ is also the standard representative of a string of type 
$\beta\alpha^{-1}\lambda$ on level $(j-1)$.
We have 
$\underline{b}^{(j)}=\underline{b}^{(j-1)}\,(\underline{c}^{(j-1)})^{a_{j-1}}\underline{d}^{(j-1)}$ and 
$\underline{l}^{(j)}=(\underline{c}^{(j-1)})^{a_{j-1}}\underline{l}^{(j-1)}$.
Moreover, $\underline{c}^{(j)}=(\underline{c}^{(j-1)})^{\xi}\underline{d}^{(j-1)}$
and $\underline{d}^{(j)}=(\underline{c}^{(j-1)})^{\zeta}\underline{d}^{(j-1)}$, where
$\{\xi,\zeta\}=\{a_{j-1},a_{j-1}+1\}$.
Using this to rewrite $(\ast)_j$ and $(\ast\ast)_j$ on level $(j-1)$, it follows that
$S$ satisfies $(\ast)_{j-1}$ and $(\ast\ast)_{j-1}$ for appropriate subpatterns of $S$ on level $(j-1)$.
Therefore, Claim 1 follows by induction.

\medskip

Let now $S=\beta\alpha^{-1}\underline{z}_0 \cdots \underline{z}_{n-1}\,\lambda$
be the standard representative of a string of type 
$\beta\alpha^{-1}\lambda$, i.e. $S$ is the standard representative of a string of type 
$\beta\alpha^{-1}\lambda$ on level $1$, and suppose
$S$ has property $(+)$ from Lemma \ref{lem:needthis}. 
Using the notation from the first paragraph of the proof, it follows that
$S=\underline{b}^{(1)}\,\underline{l}^{(1)}$, or 
$S=\underline{b}^{(1)}\,(\underline{c}^{(1)})^{a_1}\underline{l}^{(1)}$ for some positive integer
$a_1$, or
$$S=\underline{b}^{(1)}\,(\underline{c}^{(1)})^{s_1}(\underline{d}^{(1)})^{t_1}\,\cdots\,
(\underline{c}^{(1)})^{s_r}(\underline{d}^{(1)})^{t_r}(\underline{c}^{(1)})^{s_{r+1}}\,
\underline{l}^{(1)}$$
for some positive integers $r$, $s_1,\ldots,s_{r+1}$ and $t_1,\ldots,t_r$ with
$s_i=s_{r+2-i}$ and $t_j=t_{r+1-j}$ for all $i,j$ by Remark \ref{rem:string0nuB}. 
Since $S$ has property $(+)$, it cannot contain
both subpatterns $\underline{b}^{(1)}\,\underline{c}^{(1)}$ and 
$\underline{d}^{(1)}\,\underline{d}^{(1)}$. Therefore, $t_j=1$ for all $j$. 
Because $S$ cannot contain both subpatterns 
$\underline{b}^{(1)}\,(\underline{c}^{(1)})^{s_i}\underline{c}^{(1)}$ and
$\underline{d}^{(1)}\,(\underline{c}^{(1)})^{s_i}\underline{d}^{(1)}$
(resp. $S$ cannot contain both subpatterns 
$\underline{b}^{(1)}\,(\underline{c}^{(1)})^{s_1}\underline{d}^{(1)}$ and
$\underline{c}^{(1)}\,(\underline{c}^{(1)})^{s_1}\underline{c}^{(1)}$), 
we obtain $s_1\le s_i\le s_1+1$ for all $i$.
Therefore,
\begin{equation}
\label{eq:S0level1}
S=\underline{b}^{(1)}\,(\underline{c}^{(1)})^{s_1}\,\underline{d}^{(1)}\,\cdots\,
(\underline{c}^{(1)})^{s_r}\,\underline{d}^{(1)}\,(\underline{c}^{(1)})^{s_{r+1}}\,
\underline{l}^{(1)}
\end{equation}
for some positive integers $r$, $s_1,\ldots,s_{r+1}$ with $s_i\in\{s_1,s_1+1\}$
and $s_i=s_{r+2-i}$ for all $i$. Letting $a_1=s_1$, we obtain that
$$S=\underline{b}^{(1)}\,(\underline{c}^{(1)})^{a_1}\,\underline{d}^{(1)}\,
\underline{\upsilon}_0\cdots \underline{\upsilon}_{n_2-1}\,(\underline{c}^{(1)})^{a_1}\,\underline{l}^{(1)}$$
where $n_2=r-1\ge 0$ and
 $\underline{\upsilon}_i\in\{(\underline{c}^{(1)})^{a_1}\,\underline{d}^{(1)}, 
(\underline{c}^{(1)})^{a_1+1}\,\underline{d}^{(1)}\}$ for all $i\in\{0,\ldots,n_2-1\}$.  
Letting $i_1=0$ (resp. $i_1=1$) if $\underline{c}^{(1)}=\underline{x}$ (resp. 
$\underline{c}^{(1)}=\underline{y}$), we obtain
\begin{equation}
\label{eq:S0level2}
S=\underline{b}^{(2)}\,\underline{z}^{(2)}_0\cdots \underline{z}^{(2)}_{n_2-1}\,\,\underline{l}^{(2)}
\end{equation}
where $n_2\ge 0$ and 
$\underline{z}^{(2)}_i\in\{\underline{x}^{(2)},\underline{y}^{(2)}\}$ for all $i\in\{0,\ldots,n_2-1\}$.
If $n_2\ge 1$, let $\underline{c}^{(2)}=\underline{z}^{(2)}_0$ and choose $\underline{d}^{(2)}$
such that $\{\underline{c}^{(2)},\underline{d}^{(2)}\}=\{\underline{x}^{(2)},\underline{y}^{(2)}\}$.
Since $S$ satisfies $(+)$, it also satisfies $(+)_2$ by Claim 1. Therefore, 
$S=\underline{b}^{(2)}\,\underline{l}^{(2)}$, or 
$S=\underline{b}^{(2)}\,(\underline{c}^{(2)})^{a_2}\underline{l}^{(2)}$ for some positive integer
$a_2$, or
$$S=\underline{b}^{(2)}\,(\underline{c}^{(2)})^{u_1}(\underline{d}^{(2)})^{v_1}\,\cdots\,
(\underline{c}^{(2)})^{u_w}(\underline{d}^{(2)})^{v_w}(\underline{c}^{(2)})^{u_{w+1}}\,
\underline{l}^{(2)}$$
for some positive integers $w$, $u_1,\ldots,u_{w+1}$ and $v_1,\ldots,v_w$ with
$u_i=u_{w+2-i}$ and $v_j=v_{w+1-j}$ for all $i,j$. Using induction, it follows 
that there exist a positive integer $\ell$ and a choice of $(\ell-1)$-tuples 
$(i_1,\ldots, i_{\ell-1})$ and $(a_1,\ldots,a_{\ell-1})$ as in Definition \ref{def:important2}(v)
such that $S$ has one of the forms in (\ref{eq:formstring0}).

Conversely, suppose that $S$ is the standard representative of a string of type 
$\beta\alpha^{-1}\lambda$ on level $\ell$ having one of 
the forms in (\ref{eq:formstring0}). Using the notation from the first paragraph of the proof, 
this means that $S=\underline{b}^{(\ell)}\,\underline{l}^{(\ell)}$, or 
$S=\underline{b}^{(\ell)}\,(\underline{c}^{(\ell)})^{a_\ell}\underline{l}^{(\ell)}$ for some positive integer
$a_\ell$. Hence $S$ obviously has property $(+)_\ell$. By Claim 1, it follows that $S$ has
property $(+)$ from Lemma \ref{lem:needthis}. This completes the proof of Proposition \ref{prop:string0+}.
\end{proof}

\begin{prop}
\label{prop:bandstype0}
Let $B=\underline{z}_0\, \underline{z}_1\,\cdots \,\underline{z}_{n-1}$ be a standard 
representative of a band of type $0$ as in Definition $\ref{def:standardband}$. 
Then $B$ has property $(++)$ from Lemma $\ref{lem:needthisforbands}$
if and only if
there exist a positive integer $\ell$ and a choice of $(\ell-1)$-tuples
$(i_1,\ldots, i_{\ell-1})$ and  $(a_1,\ldots,a_{\ell-1})$ as in Definition $\ref{def:important2}(vi)$
such that 
\begin{equation}
\label{eq:band0+}
B\sim_{r,(0)} \underline{x}^{(\ell)}\; \mbox{ or }\; B\sim_{r,(0)} \underline{y}^{(\ell)}.
\end{equation}
Moreover, if $B$ has property $(++)$, then 
$B$ allows exactly one wrap-around, as defined in Definition $\ref{def:standardband}(iv)$.
\end{prop}

\begin{proof}
Let $\ell$ be a positive integer, let $(i_1,\ldots, i_{\ell-1})$ and $(a_1,\ldots,a_{\ell-1})$ be
$(\ell-1)$-tuples as in Definition \ref{def:important2}(vi), and let
$B$ be a standard representative of a band of type $0$ on level $\ell$ relative to these $(\ell-1)$-tuples.
As stated in Remark \ref{rem:easy}(ii), $B$ is then
also a standard representative of a band of type $0$ on level $j$
for all $1\le j\le\ell$ relative to the $(j-1)$-tuples $(i_1,\ldots, i_{j-1})$ and $(a_1,\ldots,a_{j-1})$.
Let $j\in\{1,\ldots,\ell\}$. Then
$B=\underline{z}^{(j)}_0\cdots \underline{z}^{(j)}_{n_j-1}$,
where  $\underline{z}^{(j)}_i\in\{\underline{x}^{(j)},\underline{y}^{(j)}\}$ for all $i\in\{0,\ldots,n_j-1\}$. 
For $1\le j<\ell$, define $\underline{c}^{(j)},\underline{d}^{(j)}\in \{\underline{x}^{(j)},\underline{y}^{(j)}\}$ 
such that $\underline{x}^{(j+1)}=(\underline{c}^{(j)})^{a_j}\underline{d}^{(j)}$
and $\underline{y}^{(j+1)}=(\underline{c}^{(j)})^{a_j+1}\underline{d}^{(j)}$.

For $1\le j\le \ell$, consider the following property:
\begin{enumerate}
\item[$(++)_j$] 
If there are rotated subpatterns $U,V$ of $B$ on level $j$ such that
$B\sim_{r,(0,j)} \underline{c}^{(j)}\,U\,\underline{c}^{(j)}\,V$, then for all
rotated subpatterns $W$ of $B$ on level $j$, we have 
$B\not\sim_{r,(0,j)} \underline{d}^{(j)}\,U\,\underline{d}^{(j)}\, W$. 
\end{enumerate}

\medskip

\noindent\textit{Claim $1$.} Let $B$ be a standard representative of a band of type $0$ on level 
$\ell$ as in the first paragraph of the proof. 
Then $B$ has property $(++)$ from Lemma \ref{lem:needthisforbands} if and only if 
$B$ has property $(++)_\ell$.

\medskip

\noindent
\textit{Proof of Claim $1$.}
Note that $(++)_1$ is the same as $(++)$. Moreover, $B$ does not have property $(++)_j$
for some $1\le j\le\ell$ if and only if there exist  rotated subpatterns
$U,V,W$ of $B$ on level $j$ such that 
\begin{enumerate}
\item[$(\dagger\dagger)_j$]
$B\sim_{r,(0,j)}\underline{c}^{(j)}\,U\,\underline{c}^{(j)}\,V$ and 
$B\sim_{r,(0,j)} \underline{d}^{(j)}\,U\,\underline{d}^{(j)}\, W$. 
\end{enumerate}
Suppose that $B$ does not have property $(++)_j$, i.e. $B$ satisfies $(\dagger\dagger)_j$.

If $j<\ell$, then $B$ is also a standard representative of a band of type $0$ 
on level $(j+1)$. Therefore, it follows
that the rotated subpattern $U$ in $(\dagger\dagger)_j$  has the form
$U=(\underline{c}^{(j)})^{a_j}\underline{d}^{(j)}\,\tilde{U}\,(\underline{c}^{(j)})^{a_j}$ for some 
rotated subpattern $\tilde{U}$ of $B$ on level $(j+1)$. 
Using that $\underline{x}^{(j+1)}=(\underline{c}^{(j)})^{a_j}\underline{d}^{(j)}$ and
$\underline{y}^{(j+1)}=(\underline{c}^{(j)})^{a_j+1}\underline{d}^{(j)}$, it follows that
$B$ satisfies $(\dagger\dagger)_{j+1}$ for appropriate 
rotated subpatterns of $B$ on level $(j+1)$.

If $j>1$, then $B$ is also a standard representative of a band of type $0$ 
on level $(j-1)$.
We have 
$\underline{c}^{(j)}=(\underline{c}^{(j-1)})^{\xi}\underline{d}^{(j-1)}$
and $\underline{d}^{(j)}=(\underline{c}^{(j-1)})^{\zeta}\underline{d}^{(j-1)}$, where
$\{\xi,\zeta\}=\{a_{j-1},a_{j-1}+1\}$.
Using this to rewrite $(\dagger\dagger)_j$ on level $(j-1)$, it follows that
$B$ satisfies $(\dagger\dagger)_{j-1}$ for appropriate 
rotated subpatterns of $B$ on level $(j-1)$.
Therefore, Claim 1 follows by induction.

\medskip

Let $B$ be a standard representative of a band of type 
$0$ on level $\ell$ having one of the forms in (\ref{eq:band0+}). Then $B$ 
obviously has property $(++)_\ell$. By Claim 1, it follows that $B$ has property $(++)$.

Conversely, suppose that
$B=\underline{z}_0\,\underline{z}_1\, \cdots \,\underline{z}_{n-1}$ is a standard representative 
of a band of type $0$,
i.e. $B$ is a standard representative of a band of type $0$ on level $1$, and suppose $B$ has
property $(++)$. Then $B=\underline{x}^{(1)}$, or $B=\underline{y}^{(1)}$, or 
$$B\sim_{r,(0)}(\underline{x}^{(1)})^{s_1}(\underline{y}^{(1)})^{t_1}\,\cdots\,
(\underline{x}^{(1)})^{s_m}(\underline{y}^{(1)})^{t_m}$$
for some positive integers $m$, $s_1,\ldots,s_m$ and $t_1,\ldots,t_m$.
Since $B$ has property $(++)$, it cannot contain
both rotated subpatterns $\underline{x}^{(1)}\,\underline{x}^{(1)}$ and 
$\underline{y}^{(1)}\,\underline{y}^{(1)}$. Therefore, $s_j=1$ for all $1\le j\le m$ or 
$t_j=1$ for all $1\le j\le m$.
Let $(\underline{c}^{(1)},\underline{d}^{(1)})=(\underline{x}^{(1)},\underline{y}^{(1)})$
and $e_j=s_j$ for $1\le j\le m$ if $t_j=1$ for all $j$, and let
$(\underline{c}^{(1)},\underline{d}^{(1)})=(\underline{y}^{(1)},\underline{x}^{(1)})$
and $e_j=t_j$ for $1\le j\le m$ if $s_j=1$ for all $j$. 
Then
\begin{equation}
\label{eq:B0level1}
B\sim_{r,(0)}(\underline{c}^{(1)})^{e_1}\,\underline{d}^{(1)}\,\cdots\,
(\underline{c}^{(1)})^{e_m}\,\underline{d}^{(1)}
\end{equation}
for some positive integers $m$, $e_1,\ldots,e_m$.
Let $a_1$ be minimal among $e_1,\ldots, e_m$. 
Because $B$ cannot contain both rotated subpatterns 
$\underline{c}^{(1)}\,(\underline{c}^{(1)})^{a_1}\underline{c}^{(1)}$ and
$\underline{d}^{(1)}\,(\underline{c}^{(1)})^{a_1}\underline{d}^{(1)}$ for any $j\in\{1,\ldots,m\}$, 
it follows that $e_j\le a_1+1$ for all $1\le j\le m$.
Thus $e_j\in\{a_1,a_1+1\}$ for all $1\le j\le m$, and we obtain that 
$$B\sim_{r,(0)}\underline{\upsilon}_0\,\underline{\upsilon}_1\cdots \underline{\upsilon}_{n_2-1}$$
where $n_2=m\ge 1$ and
 $\underline{\upsilon}_i\in\{(\underline{c}^{(1)})^{a_1}\,\underline{d}^{(1)}, 
(\underline{c}^{(1)})^{a_1+1}\,\underline{d}^{(1)}\}$ for all $i\in\{0,\ldots,n_2-1\}$.  
Letting $i_1=0$ (resp. $i_1=1$) if $\underline{c}^{(1)}=\underline{x}$ (resp. 
$\underline{c}^{(1)}=\underline{y}$), we obtain
\begin{equation}
\label{eq:B0level2}
B\sim_{r,(0)}\underline{z}^{(2)}_0\,\underline{z}^{(2)}_1\,\cdots \underline{z}^{(2)}_{n_2-1}
\end{equation}
where $n_2\ge 1$ and 
$\underline{z}^{(2)}_i\in\{\underline{x}^{(2)},\underline{y}^{(2)}\}$ for all $i\in\{0,\ldots,n_2-1\}$.
Since $B$ satisfies $(++)$, it follows that 
$\underline{z}^{(2)}_0\,\underline{z}^{(2)}_1\,\cdots \underline{z}^{(2)}_{n_2-1}$ 
satisfies $(++)_2$ by Claim 1. We obtain that
$B\sim_{r,(0)}\underline{x}^{(2)}$, or 
$B\sim_{r,(0)}\underline{y}^{(2)}$, or  
$$B\sim_{r,(0)}(\underline{x}^{(2)})^{u_1}(\underline{y}^{(2)})^{v_1}\,\cdots\,
(\underline{x}^{(2)})^{u_w}(\underline{y}^{(2)})^{v_w}$$
for some positive integers $w$, $u_1,\ldots,u_w$ and $v_1,\ldots,v_w$.
Using induction, it follows 
that there exist a positive integer $\ell$ and a choice of $(\ell-1)$-tuples 
$(i_1,\ldots, i_{\ell-1})$ and $(a_1,\ldots,a_{\ell-1})$ as in Definition \ref{def:important2}(vi)
such that $B$ has one of the forms in (\ref{eq:band0+}).

It remains to show that if $B$ is a standard representative of a band of type $0$ satisfying
condition $(++)$, then $B$ allows exactly one wrap-around. By what we have just proved, we
can assume that $\rho^{(0)}_u(B)=\underline{x}^{(\ell)}$ or $\rho^{(0)}_u(B)=\underline{y}^{(\ell)}$ 
for some $0\le u\le n-1$. Hence $\rho^{(0)}_u(B)$ allows exactly one wrap-around on level $\ell$.
We can rewrite $\rho^{(0)}_u(B)$ as a standard representative of a band of type $0$ on level $j$
for all $1\le j\le \ell$. Let $1< j\le \ell$, and assume by induction that
$\rho^{(0)}_u(B)$ allows exactly one wrap-around on level $j$. 
Using the notation from the first paragraph of the proof, it follows that
\begin{equation}
\label{eq:wraparoundstep}
\rho^{(0)}_u(B) = \underline{z}^{(j)}_0\,\underline{z}^{(j)}_1\,\cdots \underline{z}^{(j)}_{n_j-1}
\end{equation}
where $n_j\ge 1$ and 
$\underline{z}^{(j)}_i\in\{\underline{x}^{(j)},\underline{y}^{(j)}\}$ for all $i\in\{0,\ldots,n_j-1\}$.
Moreover, $\underline{x}^{(j)}=(\underline{c}^{(j-1)})^{a_{j-1}}\underline{d}^{(j-1)}$
and $\underline{y}^{(j)}=(\underline{c}^{(j-1)})^{a_{j-1}+1}\underline{d}^{(j-1)}$.
Therefore,
\begin{eqnarray}
\label{eq:wraparoundstep2}
\rho^{(0)}_u(B) &=& (\underline{c}^{(j-1)})^{u_0}\underline{d}^{(j-1)}\,
(\underline{c}^{(j-1)})^{u_1}\underline{d}^{(j-1)}
\,\cdots (\underline{c}^{(j-1)})^{u_{n_j-1}}\underline{d}^{(j-1)}\\
&=&  \underline{z}^{(j-1)}_0\,\underline{z}^{(j-1)}_1\,\cdots \underline{z}^{(j-1)}_{n_{j-1}-1} \nonumber
\end{eqnarray}
where $u_0,\ldots,u_{n_j-1}\in\{a_{j-1},a_{j-1}+1\}$ and
$n_{j-1}=(u_0+1)+\cdots +(u_{n_j-1}+1) = u_0+\cdots + u_{n_j-1}+n_j $.
Assume
that the unique wrap-around of $\rho^{(0)}_u(B)$ on level $j$
is at $w$, so $\underline{z}_t^{(j)} = \underline{z}_{w-t}^{(j)}$ for all $t$, where the
indices are taken modulo $n_j$.
Let $\tilde{w}=(u_0+1)+\cdots + (u_{w-1}+1)+u_w-1$. Then $\rho^{(0)}_u(B)$
allows a wrap-around on level $(j-1)$ at $\tilde{w}$.
On the other hand, if $\rho^{(0)}_u(B)$ allows a wrap-around on level $(j-1)$ at $v$, then 
(\ref {eq:wraparoundstep2}) shows that $v$ must be 
of the form $v=(u_0+1)+\cdots + (u_{\tilde{v}-1}+1)+u_{\tilde{v}}-1$ for some $\tilde{v}\in
\{0,\ldots, n_j-1\}$. Hence this implies that $u_t=u_{\tilde{v}-t}$ for all $t$, and thus
$\underline{z}_t^{(j)} = \underline{z}_{\tilde{v}-t}^{(j)}$ for all 
$t$, where the indices are taken modulo $n_j$. Therefore, the uniqueness of the wrap-around
on level $j$ leads to the uniqueness of the wrap-around on level $(j-1)$.
Using induction, this completes the proof of Proposition \ref{prop:bandstype0}.
\end{proof}


\section{Appendix: The representation theory of  $\Lambda_0$}
\label{s:prelimstring}
\setcounter{equation}{0}

As in \S \ref{s:particular}, let $k$ be an algebraically closed field of arbitrary characteristic and let 
$\Lambda_0$ be the basic $k$-algebra $\Lambda_0=kQ/I$ where $Q$ and $I$ are as in Figure 
\ref{fig:algebra}. In this appendix, we describe the representation theory of $\Lambda_0$.
Since $\Lambda_0$ is a special biserial algebra, we can use the results from \cite{buri,krau}.
It follows from \cite{buri} that all indecomposable non-projective 
$\Lambda_0$-modules are either  string or band modules.

\subsection{String and band modules for $\Lambda_0$}
\label{ss:stringband}

For each arrow $\alpha,\lambda,\xi,\delta,\rho,\beta$ in $Q$, we define a formal inverse
$\alpha^{-1},\lambda^{-1},\xi^{-1},\delta^{-1},\rho^{-1},\beta^{-1}$, respectively, with starting vertices 
$s(\alpha^{-1})=0=s(\lambda^{-1})$, $s(\rho^{-1})=1=s(\beta^{-1})$ and $s(\xi^{-1})=2=s(\delta^{-1})$
and end vertices $e(\alpha^{-1})=0=e(\beta^{-1})$, $e(\rho^{-1})=1=e(\delta^{-1})$ and 
$e(\xi^{-1})=2=e(\lambda^{-1})$.
A word $w$ is a sequence $w_1\cdots w_n$, where $w_i$ is either an arrow or a formal inverse such 
that $s(w_i)=e(w_{i+1})$ for $1\leq i \leq n-1$. Define $s(w)=s(w_n)$, $e(w)=e(w_1)$ and 
$w^{-1}=w_n^{-1}\cdots w_1^{-1}$. There are also empty words $1_0$, $1_1$ and $1_2$ of length 
$0$ with $e(1_u)=u=s(1_u)$ and $(1_u)^{-1}=1_u$ for $u\in\{0,1,2\}$. 
Denote the set of all words by $\mathcal{W}$, and the set of all 
non-empty words $w$ with $e(w)=s(w)$ by ${\mathcal{W}}_r$. 
Let $J=\{\alpha^2,\rho^2,\xi^2,\alpha\lambda,\lambda\xi,
\xi\delta,\delta\rho,\rho\beta,\beta\alpha,
\lambda\delta\beta,\beta\lambda\delta,\delta\beta\lambda\}$.

\begin{dfn}
\label{def:strings}
Let $\sim_s$ be the equivalence relation on $\mathcal{W}$ with $w\sim_s w'$ if and only if $w=w'$ or 
$w^{-1}=w'$. Then \emph{strings} are representatives $w\in{\mathcal{W}}$ of the equivalence classes under 
$\sim_s$ with the following property: Either $w=1_u$ for $u\in\{0,1,2\}$, or $w=w_1\cdots w_n$ where 
$w_i \neq w_{i+1}^{-1}$ for $1\leq i\leq n-1$ and no subword of $w$ or its formal inverse belongs to 
$J$.

Let $C=w_1\cdots w_n$ be a string of length $n$. Then there exists an
indecomposable $\Lambda_0$-module $M(C)$,
called the \emph{string module} corresponding to the string $C$, which can be described as follows.
There is an ordered $k$-basis $\{z_0,z_1,\ldots, z_n\}$ of $M(C)$ such that the action of 
$\Lambda_0$ on $M(C)$ 
is given by the following representation $\varphi_C:\Lambda_0\to\mathrm{Mat}(n+1,k)$. 
Let $v(i)=e(w_{i+1})$ for $0\leq i\leq n-1$ and $v(n)=s(w_n)$. Then for each vertex $u\in\{0,1,2\}$ and for 
each arrow $\zeta\in\{\alpha,\lambda,\xi,\delta,\rho,\beta\}$ in $Q$
$$\varphi_C(u)(z_i) = \left\{ \begin{array}{c@{\quad,\quad}l}
z_i & \mbox{if $v(i)=u$}\\ 0 & \mbox{else} \end{array} \right\}\; \mbox{ and } \;
\varphi_C(\zeta)(z_i) = \left\{ \begin{array}{c@{\quad,\quad}l}
z_{i-1} & \mbox{if $w_i=\zeta$}\\ z_{i+1} & \mbox{if $w_{i+1}=\zeta^{-1}$}\\
0 & \mbox{else}
\end{array} \right\} .$$
We  call $\varphi_C$ the \emph{canonical representation} and $\{z_0,z_1,\ldots,z_n\}$ a 
\emph{canonical $k$-basis} for $M(C)$ relative to the representative $C$. Note that 
$M(C)\cong M(C^{-1})$. 

The string modules for the empty words are isomorphic to the simple $\Lambda_0$-modules,
namely $M(1_0)\cong S_0$, $M(1_1)\cong S_1$ and $M(1_2)\cong S_2$.
\end{dfn}

\begin{dfn}
\label{def:bands}
Let $w=w_1\cdots w_n\in {\mathcal{W}}_r$. Then, for $0\leq i\leq n-1$, the \emph{$i$-th rotation} of $w$ is 
defined to be the word $\rho_i(w)=w_{i+1}\cdots w_n w_1 \cdots w_i$. Let $\sim_r$ be the 
equivalence relation on ${\mathcal{W}}_r$ such that
$w\sim_r w'$ if and only if $w=\rho_i(w')$ for some $i$ or $w^{-1}=\rho_j(w')$ for some $j$. 
Then \emph{bands} are representatives $w\in {\mathcal{W}}_r$ of the equivalence classes under 
$\sim_r$ with the following property: 
$w=w_1\cdots w_n$, $n\ge 1$, with $w_i\neq w_{i+1}^{-1}$ and $w_n\neq w_1^{-1}$, such that 
$w$ is not a power of a smaller word, and, for all positive integers $m$, no subword of $w^m$ 
or its formal inverse belongs to $J$.

Let $B=w_1\cdots w_n$ be a band of length $n$. Then for each integer $m>0$ and each 
$\mu\in k^*$ there exists an indecomposable $\Lambda_0$-module $M(B,\mu,m)$ which is
called the \emph{band module} corresponding to the band $B$, $\mu$ and $m$, which can be
described as follows. There is an ordered $k$-basis 
$$\{z_{0,1},z_{0,2},\ldots, z_{0,m},z_{1,1},\ldots,z_{1,m},\ldots, z_{n-1,1},\ldots,z_{n-1,m}\}$$
of $M(B,\mu,m)$ such that the action of $\Lambda_0$ on $M(B,\mu,m)$ 
is given by the following representation $\varphi_{B,\mu,m}:\Lambda_0\to\mathrm{Mat}(n\cdot m,k)$. 
Let $v(i)=e(w_{i+1})$ for $0\leq i\leq n-1$. Then for each vertex $u\in\{0,1,2\}$ and for 
each arrow $\zeta\in\{\alpha,\lambda,\xi,\delta,\rho,\beta\}$ in $Q$, we have
\begin{eqnarray*}
\varphi_{B,\mu,m}(u)(z_{i,j}) &=& \left\{ \begin{array}{c@{\quad,\quad}l}
z_{i,j} & \mbox{if $v(i)=u$}\\ 0 & \mbox{else} \end{array} \right\}\; \mbox{ and } \;\\[1ex]
\varphi_{B,\mu,m}(\zeta)(z_{i,j})& =& \left\{ \begin{array}{c@{\quad,\quad}l}
\mu\, z_{0,j}+z_{0,j+1} & \mbox{if $w_i=\zeta$ and $i= 1$}\\
z_{i-1,j} & \mbox{if $w_i=\zeta$ and $i\neq 1$}\\ 
\mu^{-1} z_{1,j}+z_{1,j+1}&  \mbox{if $w_{i+1}=\zeta^{-1}$ and $i=0$}\\
z_{i+1,j} & \mbox{if $w_{i+1}=\zeta^{-1}$ and $i\neq 0$}\\
0 & \mbox{else}
\end{array} \right\} 
\end{eqnarray*}
for all $0\le i\le n-1$, $1\le j \le m$,
where $z_{0,m+1}=0=z_{1,m+1}$ and $z_{n,j}=z_{0,j}$ for all $j$.
We  call $\varphi_{B,\mu,m}$ the \emph{canonical representation} and 
$\{z_{0,1},z_{0,2},\ldots, z_{0,m},z_{1,1},\ldots,z_{1,m},\ldots, z_{n-1,1},\ldots,z_{n-1,m}\}$ 
a \emph{canonical $k$-basis} for $M(B,\mu,m)$ relative to the representative $B$.
Note that we have for all $0\le i,j\le n-1$,
$$M(B,\mu,m)\cong M(\rho_i(B),\mu,m)\cong M(\rho_j(B)^{-1},\mu^{-1},m).$$ 
\end{dfn}

\subsection{The stable Auslander-Reiten quiver of $\Lambda_0$}
\label{ss:stableARquiver}

Each component of the stable Auslander-Reiten quiver of $\Lambda_0$ consists either entirely of string 
modules or entirely of band modules. The band modules all lie in $1$-tubes. The components
consisting of string modules are two $3$-tubes and infinitely many non-periodic 
components of type $\mathbb{Z}A_\infty^\infty$. 

The irreducible morphisms between string modules can be described using hooks and cohooks.  
For our algebra $\Lambda_0$, these are defined as follows. Let $\mathcal{M}$ be the set of 
maximal directed strings, i.e. $\mathcal{M}=\{\alpha,\rho,\xi,\lambda\delta,\delta\beta,
\beta\lambda\}$.

\begin{dfn}
\label{def:arcomps}
Let $S$ be a string.
We say that $S$ \emph{starts on a peak} (resp. \emph{starts in a deep}) if $S=S'C$ (resp. $S=S'C^{-1}$)
for some string $C$ in $\mathcal{M}$. Dually, we say that $S$ \emph{ends on a peak} (resp. \emph{ends in a deep})
if $S=D^{-1}S''$ (resp. $S=DS''$) for some string $D$ in $\mathcal{M}$.

If $S$ does not start on a peak (resp. does not start in a deep), there is a unique arrow $\zeta$ and a 
unique $M\in\mathcal{M}$ such that $S_h=S\zeta M^{-1}$ (resp. $S_c=S\zeta^{-1}M$) is a string. 
We say $S_h$ (resp. $S_c$) is obtained from $S$ by adding a \emph{hook} (resp. a \emph{cohook}) on 
the right side.

Dually, if $S$ does not end on a peak (resp. does not end in a deep), there is a unique arrow $\xi$ 
and a unique $N\in\mathcal{M}$ such that ${}_hS=N\xi^{-1}S$ (resp. ${}_cS=N^{-1}\xi S$)
is a string.  We say ${}_hS$ (resp. ${}_cS$) is obtained from $S$ by adding a \emph{hook} (resp. a \emph{cohook}) 
on the left side.
\end{dfn}

All irreducible morphisms between string modules are either canonical injections
$M(S)\to M(S_h)$, $M(S)\to M({}_hS)$,
or canonical projections
$M(S_c)\to M(S)$, $M({}_cS)\to M(S)$.

In particular, since none of the projective $\Lambda_0$-modules is uniserial, we get the following result.
Suppose $M(S)$ is a string module of minimal length such that $M(S)$ belongs to 
a component of the stable Auslander-Reiten
quiver of $\Lambda_0$ of type $\mathbb{Z}A_\infty^\infty$. 
Then near $M(S)$ the stable Auslander-Reiten component looks as in Figure \ref{fig:arbcomp}.
\begin{figure}[ht] \hrule \caption{\label{fig:arbcomp} The stable Auslander-Reiten component 
near $M(S)$.}
$$\xymatrix @-1.2pc{
&&&&&&\\
&M({}_{cc}S)\ar[rd]\ar@{.}[lu]\ar@{.}[ld]\ar@{.}[ru]&&
M({}_cS_h)\ar[rd]\ar@{.}[ru]\ar@{.}[lu]
&&M(S_{hh})\ar@{.}[ru]\ar@{.}[rd]\ar@{.}[lu]&\\
&&M({}_{c}S)\ar[rd]\ar[ru]&&M(S_{h})\ar[rd]\ar[ru]&&\\
&M({}_cS_c)\ar[ru]\ar[rd]\ar@{.}[lu]\ar@{.}[ld]&&M(S)\ar[ru]\ar[rd]&&M({}_hS_h)\ar@{.}[ru]\ar@{.}[rd]&\\
&&M(S_{c})\ar[ru]\ar[rd]\ar@{.}[ld]&&M({}_{h}S)\ar[ru]\ar[rd]&&\\
&M(S_{cc})\ar[ru]\ar@{.}[ld]\ar@{.}[lu]\ar@{.}[rd]&&
M({}_hS_c)\ar[ru]\ar@{.}[rd]\ar@{.}[ld]
&&M({}_{hh}S)\ar@{.}[rd]\ar@{.}[ru]\ar@{.}[ld]&\\
&&&&&&\\&&&&&&
}$$
\hrule
\end{figure}
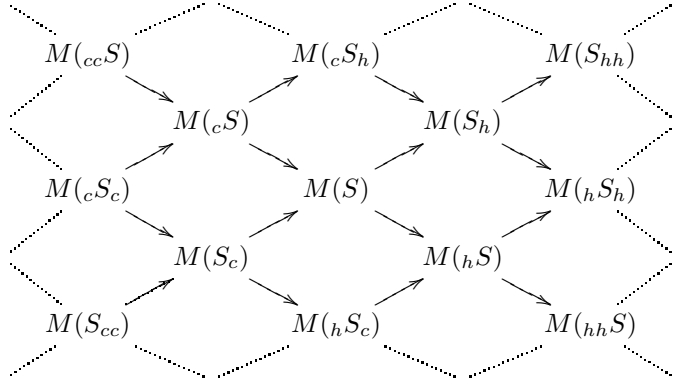

\subsection{Homomorphisms between string and band modules for $\Lambda_0$}
\label{ss:homs}

In \cite{krau}, all homomorphisms between string and band modules have been determined.
The following remark describes the homomorphisms between string modules using the
canonical bases defined in Definition \ref{def:strings}.

\begin{rem}
\label{rem:stringhoms}
Let $M(S)$ (resp. $M(T)$) be a string module for $\Lambda_0$ with a canonical 
$k$-basis $\{x_u\}_{u=0}^m$ 
(resp. $\{y_v\}_{v=0}^n$) relative to the representative $S$ (resp. $T$).
Suppose $C$ is a string such that
\begin{enumerate}
\item[(i)] $S\sim_s S'CS''$ with ($S'$ of length $0$ or $S'=\hat{S}'\xi_1$) and ($S''$ of length $0$ or 
$S''=\xi_2^{-1} \hat{S}''$), where $S',\hat{S}',S'',\hat{S}''$ are strings and $\xi_1,\xi_2$ are arrows in $Q$; 
and
\item[(ii)] $T\sim_sT'CT''$ with ($T'$ of length $0$ or $T'=\hat{T}'\zeta_1^{-1} $) and ($T''$ of length $0$ or 
$T''=\zeta_2 \hat{T}''$), where $T',\hat{T}',T'',\hat{T}''$ are strings and $\zeta_1,\zeta_2$ are arrows in $Q$.
\end{enumerate} 
Then by \cite{krau} there exists a non-zero $\Lambda_0$-module homomorphism 
$\sigma_C:M(S)\to M(T)$ which factors through $M(C)$ and which sends 
each element of $\{x_u\}_{u=0}^m$ either to zero or to an element of $\{y_v\}_{v=0}^n$,
according to the relative position of $C$ in $S$ and $T$, respectively.
If e.g. $S=s_1s_2\cdots s_m$, $T=t_1t_2\cdots t_n$, and $C=s_{i+1}s_{i+2}\cdots s_{i+\ell} = t_{j+\ell}^{-1}t_{j+\ell-1}^{-1}\cdots t_{j+1}^{-1}$,
then 
$$\sigma_C(x_{i+t})=y_{j+\ell-t} \mbox{ for } 0\le t\le \ell, \mbox{ and } \sigma_C(x_u)=0\mbox{ for all other $u$.}$$
We call $\sigma_C$ a \emph{canonical homomorphism} from $M(S)$ to $M(T)$.
Note that there may be several choices for $S',S''$ (resp. $T',T''$) in (i) (resp. (ii)). In other words, there 
may be several $k$-linearly independent canonical homomorphisms factoring through $M(C)$. 
By \cite{krau}, every $\Lambda_0$-module homomorphism $\sigma:M(S)\to M(T)$ can be written 
as a unique $k$-linear combination of canonical
homomorphisms which factor through string modules corresponding to strings $C$ satisfying 
(i) and (ii). 
\end{rem}

It follows from \cite{krau} that if $B$ is a band, $\mu\in k^*$ and $n\ge 2$ is an  
integer, then $\underline{\mathrm{End}}_{\Lambda_0}(M(B,\mu,n))$ has $k$-dimension
at least $2$. The following remark describes the homomorphisms between band modules 
of the form $M(B,\mu,1)$
using the canonical bases defined in Definition \ref{def:bands}.

\begin{rem}
\label{rem:bandendohelp}
Let  $B,\tilde{B}$ be bands  and let $\mu,\tilde{\mu}\in k^*$. Let $M_{B,\mu}=M(B,\mu,1)$ 
(resp. $M_{\tilde{B},\tilde{\mu}}=M(\tilde{B},\tilde{\mu},1)$) with a
canonical $k$-basis $\{z_{u,1}\}_{0\le u\le m-1}$ (resp. $\{\tilde{z}_{v,1}\}_{0\le v\le n-1}$)
relative to the representative $B$ (resp. $\tilde{B}$).
Suppose $S$ is a string such that 
\begin{enumerate}
\item[(i)] $B\sim_rST_1$ with $T_1=\xi_1^{-1} T'_1\xi_2$, where $T_1,T'_1$ are strings 
and $\xi_1,\xi_2$ are arrows in $Q$; and 
\item[(ii)] $\tilde{B}\sim_r ST_2$ with $T_2=\zeta_1 T'_2\zeta_2^{-1}$, where $T_2,T'_2$ are 
strings and $\zeta_1,\zeta_2$ are arrows in $Q$.
\end{enumerate} 
Then by \cite{krau} there exists a non-zero $\Lambda_0$-module homomorphism 
$\tau_S:M_{B,\mu}\to M_{\tilde{B},\tilde{\mu}}$ which factors 
through $M(S)$ and which sends each element of $\{z_{u,1}\}_{0\le u\le m-1}$ either to zero
or to an element of $\{\tilde{z}_{v,1}\}_{0\le v\le n-1}$, according to the relative position of $S$ in $B$
and $\tilde{B}$, respectively.
Suppose e.g. $B=w_1\cdots w_m$, $\tilde{B}=\tilde{w}_1\cdots \tilde{w}_n$ 
and $S=w_{i+1}w_{i+2}\cdots w_{i+\ell}=
\tilde{w}_{j+\ell}^{-1}\tilde{w}_{j+\ell-1}^{-1}\cdots \tilde{w}_{j+1}^{-1}$, 
where $w_i=\xi_2$, $w_{i+\ell+1}=\xi_1^{-1}$, $\tilde{w}_j^{-1}=\zeta_1$, $\tilde{w}_{j+\ell+1}^{-1}=\zeta_2^{-1}$.
Then
$$\tau_S(z_{i+t,1})=\tilde{z}_{j+\ell-t,1} \mbox{ for } 0\le t\le \ell, \mbox{ and } \tau_S(z_{u,1})=0\mbox{ for all other $u$.}$$
We call $\tau_C$ a \emph{canonical homomorphism} from $M_{B,\mu}$ to $M_{\tilde{B},\tilde{\mu}}$
\emph{of string type $S$}. Note 
that there may be several choices for $T_1$ (resp. $T_2$) in (i) (resp. (ii)). In 
other words, there may be several $k$-linearly independent canonical 
homomorphisms of string type $S$. By \cite{krau}, if 
$M_{B,\mu}$ and $M_{\tilde{B},\tilde{\mu}}$ are not isomorphic,
then 
every $\Lambda_0$-module homomorphism $\tau:M_{B,\mu}\to M_{\tilde{B},\tilde{\mu}}$ can be written 
as a unique $k$-linear combination of canonical homomorphisms of string type $S$ 
for suitable choices of strings $S$ satisfying (i) and (ii). If $B=\tilde{B}$ and $\mu=\tilde{\mu}$, then every
$\Lambda_0$-module endomorphism of $M_{B,\mu}$ can be
written as a unique $k$-linear combination of the identity homomorphism and of 
canonical endomorphisms of string type $S$ for suitable 
choices of strings $S$ satisfying (i) and (ii). 
\end{rem}


\end{document}